\documentclass{amsart}

\usepackage{blkarray}
\usepackage{amssymb,amsmath,amsrefs}
\usepackage{wrapfig}
\usepackage{subcaption}
\usepackage{enumerate}
\usepackage{wasysym}  
\usepackage{stmaryrd}
\usepackage{listings}
\usepackage{verbatim}
\usepackage[dvipsnames]{xcolor}
\usepackage{pgfplots}
\usepackage{tikz}
\usepackage{tikz-3dplot}
\usepackage{tikz-cd}
\pgfplotsset{compat=1.17}
\usetikzlibrary{shapes,positioning}
\usepackage{setspace}
\usepackage{hyperref}
\hypersetup{
    colorlinks = false,
    linkbordercolor = {white},
}

\makeatletter
\DeclareFontFamily{OMX}{MnSymbolE}{}
\DeclareSymbolFont{MnLargeSymbols}{OMX}{MnSymbolE}{m}{n}
\SetSymbolFont{MnLargeSymbols}{bold}{OMX}{MnSymbolE}{b}{n}
\DeclareFontShape{OMX}{MnSymbolE}{m}{n}{
    <-6>  MnSymbolE5
   <6-7>  MnSymbolE6
   <7-8>  MnSymbolE7
   <8-9>  MnSymbolE8
   <9-10> MnSymbolE9
  <10-12> MnSymbolE10
  <12->   MnSymbolE12
}{}
\DeclareFontShape{OMX}{MnSymbolE}{b}{n}{
    <-6>  MnSymbolE-Bold5
   <6-7>  MnSymbolE-Bold6
   <7-8>  MnSymbolE-Bold7
   <8-9>  MnSymbolE-Bold8
   <9-10> MnSymbolE-Bold9
  <10-12> MnSymbolE-Bold10
  <12->   MnSymbolE-Bold12
}{}

\let\llangle\@undefined
\let\rrangle\@undefined
\DeclareMathDelimiter{\llangle}{\mathopen}%
                     {MnLargeSymbols}{'164}{MnLargeSymbols}{'164}
\DeclareMathDelimiter{\rrangle}{\mathclose}%
                     {MnLargeSymbols}{'171}{MnLargeSymbols}{'171}
\makeatother

\newtheorem{thm}{Theorem}[section]

\newtheorem{defn}[thm]{Definition}
\newtheorem{prop}[thm]{Proposition}

\newtheorem{lemma}[thm]{Lemma}
\newtheorem{lem}[thm]{Lemma}

\theoremstyle{remark}
\newtheorem{ex}[thm]{Example}
\newtheorem{rem}[thm]{Remark}

\newcommand{\K}{\mathbb{K}}

\renewcommand{\H}{\mathcal{H}}

\renewcommand{\leq}{\leqslant} 
\renewcommand{\geq}{\geqslant}
\renewcommand{\bar}{\overline}

\newcommand{\ch}{\mathtt{C}}

\DeclareMathOperator{\Sym}{Sym}

\DeclareMathOperator{\Der}{Sig}

\DeclareMathOperator{\D}{\mathfrak{F}}

\usepackage{xcolor}

\definecolor{cerulean}{rgb}{0,.48,.65} 
\definecolor{magenta}{rgb}{.5,0,.5} 
\definecolor{dred}{rgb}{.5,0,0} 
\definecolor{green}{rgb}{0,.5,0} 
\definecolor{blue}{rgb}{0,0,0.5} 
\definecolor{black}{rgb}{0,0,0} 
\definecolor{dgreen}{rgb}{0,.3,0} 
\definecolor{vdred}{rgb}{.3,0,0} 
\definecolor{red}{rgb}{1,0,0} 
\definecolor{salmon}{rgb}{0.98,0.50,0.45} 
\definecolor{gray}{rgb}{.5,.5,.5} 
\definecolor{seagreen}{rgb}{0.13,0.70,0.67} 
\definecolor{chartreuse}{rgb}{0.40,0.80,0.00}
\definecolor{cornflower}{rgb}{0.39,0.58,0.93} 
\definecolor{gold}{rgb}{0.80,0.68,0.00}
\definecolor{cornellred}{rgb}{0.7, 0.11, 0.11} 
\definecolor{csugreen}{rgb}{0.12,0.31,0.17}
\definecolor{csugold}{rgb}{0.78,0.76,0.45}
\definecolor{bucknellorange}{rgb}{0.91,0.46,0.13}

\tikzset{pics/.cd,
cube/.style args={#1/#2/#3/#4}{code={
\coordinate (O) at (0,0,0);
\coordinate (A) at (0,#2,0);
\coordinate (B) at (0,#2,#3);
\coordinate (C) at (0,0,#3);
\coordinate (D) at (#1,0,0);
\coordinate (E) at (#1,#2,0);
\coordinate (F) at (#1,#2,#3);
\coordinate (G) at (#1,0,#3);
\draw[black,fill=black!80] (O) -- (C) -- (G) -- (D) -- cycle;
\draw[black,fill=black!30] (O) -- (A) -- (E) -- (D) -- cycle;
\draw[black,fill=black!10] (O) -- (A) -- (B) -- (C) -- cycle;
\draw[black,fill=black!20,opacity=0.8] (D) -- (E) -- (F) -- (G) -- cycle;
\draw[black,fill=black!20,opacity=0.6] (C) -- (B) -- (F) -- (G) -- cycle;
\draw[black,fill=black!20,opacity=0.8] (A) -- (B) -- (F) -- (E) -- cycle;
\node at (0.5*#1,0.5*#2,0.5*#3) {#4};
}}}
\tikzset{pics/.cd,
   tcube/.style args={#1/#2/#3/#4/#5}{
      code={
         \coordinate (O) at (0,0,0);
         \coordinate (A) at (0,#2,0);
         \coordinate (B) at (0,#2,#3);
         \coordinate (C) at (0,0,#3);
         \coordinate (D) at (#1,0,0);
         \coordinate (E) at (#1,#2,0);
         \coordinate (F) at (#1,#2,#3);
         \coordinate (G) at (#1,0,#3);
         \draw[black,fill=black!5,opacity=0.6] (O) -- (A) -- (B) -- (C)-- cycle;
         \draw[black,fill=black!5,opacity=1] (O) -- (A) -- (E) -- (D)-- cycle;
         \draw[black,fill=black!5,opacity=0.8] (O) -- (D) -- (G) -- (C)-- cycle;
         \node at (0.5*#1,0.6*#2,0.5*#3) {{\small #5}};

         \draw[dotted] (E) -- (F) -- (G);
         \draw[dotted] (B) -- (F);
      }
   }
}
\tikzset{pics/.cd,
ccube/.style args={#1/#2/#3/#4}{code={
\coordinate (O) at (0,0,0);
\coordinate (A) at (0,#2,0);
\coordinate (B) at (0,#2,#3);
\coordinate (C) at (0,0,#3);
\coordinate (D) at (#1,0,0);
\coordinate (E) at (#1,#2,0);
\coordinate (F) at (#1,#2,#3);
\coordinate (G) at (#1,0,#3);
\draw[black,fill=#4!80] (O) -- (C) -- (G) -- (D) -- cycle;
\draw[black,fill=#4!30] (O) -- (A) -- (E) -- (D) -- cycle;
\draw[black,fill=#4!10] (O) -- (A) -- (B) -- (C) -- cycle;
\draw[black,fill=#4!20,opacity=0.8] (D) -- (E) -- (F) -- (G) -- cycle;
\draw[black,fill=#4!20,opacity=0.6] (C) -- (B) -- (F) -- (G) -- cycle;
\draw[black,fill=#4!20,opacity=0.8] (A) -- (B) -- (F) -- (E) -- cycle;
}}}
\tikzset{pics/.cd,
linecube/.style args={#1/#2/#3}{code={
\coordinate (O) at (0,0,0);
\coordinate (A) at (0,#2,0);
\coordinate (B) at (0,#2,#3);
\coordinate (C) at (0,0,#3);
\coordinate (D) at (#1,0,0);
\coordinate (E) at (#1,#2,0);
\coordinate (F) at (#1,#2,#3);
\coordinate (G) at (#1,0,#3);
\draw[black] (O) -- (D) -- (E) -- (F) -- (B) -- (C) -- cycle;
\draw[dotted] (D) -- (G) -- (F);
\draw[dotted] (C) -- (G);
}}}

\tikzset{pics/.cd,
lwing/.style args={#1/#2/#3/#4}{code={
\coordinate (O) at ( 0, 0, 0);
\coordinate (A) at ( 0, 0,#3);
\coordinate (B) at ( 0,#2,#3);
\coordinate (C) at ( 0,#2, 0);
\draw[black,fill=red!20] (O) -- (A) -- (B) -- (C) -- cycle;
\node at (0,0.5*#2,0.5*#3) {#4};
}}}
\tikzset{pics/.cd,
mwing/.style args={#1/#2/#3/#4}{code={
\coordinate (O) at ( 0, 0, 0);
\coordinate (A) at (#1, 0, 0);
\coordinate (B) at (#1,#2, 0);
\coordinate (C) at ( 0,#2, 0);
\draw[black,fill=green!20] (O) -- (A) -- (B) -- (C) -- cycle;
\node at (0.5*#1,0.5*#2,0) {#4};
}}}
\tikzset{pics/.cd,
rwing/.style args={#1/#2/#3/#4}{code={
\coordinate (O) at ( 0, 0, 0);
\coordinate (A) at (#1, 0, 0);
\coordinate (B) at (#1, 0,#3);
\coordinate (C) at ( 0, 0,#3);
\draw[black,fill=blue!20] (O) -- (A) -- (B) -- (C) -- cycle;
\node at (0.5*#1,0,0.5*#3) {#4};
}}}

\pgfmathsetmacro{\xx}{0.5}
\tikzset{pics/.cd,
    nilcube/.style args={#1/#2}{
        code={
            \pic at (0,-#1,0) {ccube={#1/#1/#1/#2}};
            \pic at (#1,0,0) {ccube={#1/#1/#1/#2}};
            \pic at (0,0,#1) {ccube={#1/#1/#1/#2}};
            \pic at (0,#1,0) {linecube={2*#1/-2*#1/2*#1}};
        }
    }
}

\DeclareDocumentCommand \braket { o m m } {
        \IfNoValueTF {#3} {
            \left\langle #1\, \middle|\, #2 \right\rangle
        }{
            \left\langle #2\,\middle|_{#1}\, #3 \right\rangle
        }
    }

\usepackage{algorithm}
\usepackage[noend]{algpseudocode}

\title{New linear invariants of hypergraphs}

\author{Peter A. Brooksbank}
\thanks{${}^*$ Corresponding author}
\address{Department of Mathematics, Bucknell University, Lewisburg, PA 17837, USA}
\email{pbrooksb@bucknell.edu}
\author{Clara R. Chaplin}
\address{Department of Mathematics, Bucknell University, Lewisburg, PA 17837, USA}
\email{crc035@bucknell.edu}

\date{\today}

\begin{document}

\begin{abstract}
We introduce a parameterized family of invariants for $\ell$-uniform 
hypergraphs.
To each $\K$-linear transformation $T:\K^{\ell}\to \K^r$ we associate a function 
$\Der(-,T)$ that maps $\ell$-uniform hypergraphs to $\K$-vector spaces. 
Given an $\ell$-uniform hypergraph $\H=(V,E)$, we 
use $\Der(\H,T)$ to define an equivalence relation $\equiv_T$ on $V$ 
called \emph{$T$-fusion}, which determines 
a quotient hypergraph $\mathfrak{F}(\H,T)$ called the \emph{$T$-frame of $\H$}.
We show that the map $U:\K^{\ell}\to \K$, where 
$U(\lambda)=\lambda(1)+\cdots+\lambda(\ell)$,
is universal in that $\Der(\H,T)$ 
embeds in $\Der(\H,U)$, and
$U$-fusion refines $T$-fusion for any $T:\K^{\ell}\to\K^r$.
We further show that $\mathfrak{F}(\mathfrak{F}(\H,U),U)=\mathfrak{F}(\H,U)$ 
for any $\ell$-uniform hypergraph $\H$, so
$\mathfrak{F}(-,U)$ is a closure function on the set of 
$\ell$-uniform hypergraphs. 
We explore the properties of this one-time simplification 
of a hypergraph.
\end{abstract}

\maketitle

\section{Introduction}
\label{sec:intro}

Invariants of the adjacency matrix of a graph are commonly deployed 
to gain insight into the structure of the graph. 
Indeed spectral graph theory studies 
the properties of graphs through the lens of such invariants~\cite{Chung-Book}.
Recent works such 
as~\citelist{\cite{CD-Hypergraph-Spectra}\cite{PZ-hypergraph-tensor}} 
have sought to generalize spectral methods to hypergraphs. 
It has been helpful in such efforts to identify the given hypergraph with a 
multilinear product (\textit{tensor}), and   
the origins of the current work lie in a 
recent spectral theory of tensors~\cite{FMW}.
New families of linear invariants of tensors were introduced 
in~\cite{Strata} to detect, in a given tensor, hidden patterns 
belonging to a broad range of parameterized 
families. Our objective is 
to adapt those invariants to hypergraphs, and then use them to
explore combinatorial features. 


Our main contribution is to define 
an equivalence relation called \textit{fusion}
on the vertices of a hypergaph. In 
that sense our concern is similar to works such 
as~\citelist{\cite{BS-Symmetry}\cite{BS-Building-Blocks}} that study 
hypergraph operators 
compatible with equivalence relations. However, our process 
works in reverse. Instead of starting with a \textit{prescribed} 
equivalence on the vertices of a hypergraph $\H$, 
to each linear transformation $T$ we compute
operators $\Der(\H,T)$ called \textit{$T$-signals} 
and use them to \textit{define} an equivalence called \textit{$T$-fusion}. 
We show that the transformation $U:\K^{\ell}\to \K$, where 
$U(\lambda)=\lambda(1)+\cdots+\lambda(\ell)$,
is universal in that, for any $T:\K^{\ell}\to\K^r$, the $\K$-space 
$\Der(\H,T)$ embeds in $\Der(\H,U)$. As a consequence,
$U$-fusion is the most refined of 
the fusion relations.

We then study the function that maps 
$\H$ to the quotient $\D(\H)$, called the \textit{$U$-frame} of $\H$, 
defined by that equivalence.
One outcome of this study is that $\D(-)$ is a closure function 
on the set of finite connected $\ell$-uniform hypergraphs 
(Theorem~\ref{thm:idempotent}). 
Another outcome is a polynomial-time algorithm that, given 
such a hypergraph $\H$, computes the quotient $\D(\H)$.
Thus, our technique offers an efficiently computible, one-time 
simplification of a given hypergraph. 

We have made a Python implementation of the techniques described 
in the paper publicly available on GitHub~\cite{Python}.
In Section~\ref{sec:applications} we report on some experiments carried out 
with our programs. 
The final section of the paper briefly 
discusses two families of hypergraphs whose members are distinguished by 
the full tensor invariants in~\cite{Strata}, but not
by our restricted invariants. This suggests there is much more 
to learn from and about this nascent spectral theory of hypergraphs.

\section{Hypergraphs}
\label{sec:background}
In this section we summarize just the hypergraph concepts and 
notation we need throughout, and refer to~\cite{Berge-Book} for a thorough 
treatment of the subject.
\smallskip

Denote by $Y^X$ the set of
functions from $X$ to $Y$. For $f\in Y^X$ and $x\in X$, we 
denote the image of $x$ under $f$ by $f(x)$. However, when 
$Y$ is itself a set of functions, 
we find it convenient to adopt the notation $x\mapsto f_x$, so that 
$f_x(-)$ can be used for evaluation of $f_x$. Throughout the paper, 
$\ell\geq 3$ is a fixed integer, and 
\[
[\ell]=\{1,\ldots,\ell\}.
\]
Write $Y^{[\ell]}$ as $Y^{\ell}$ for short, and
let $\Sym(\ell)$ be the group of permutations of $[\ell]$.
For $f\in Y^{\ell}$ and $\sigma\in\Sym(\ell)$, define
$f^{\sigma}\in Y^{\ell}$, where $f^{\sigma}(a)=f(a^{\sigma})$ 
for $a\in [\ell]$. There is an equivalence relation on $Y^{\ell}$ 
where $f$ is equivalent to $g$ if there exists $\sigma\in\Sym(\ell)$
such that $g=f^{\sigma}$. The equivalence class 
\begin{align*}
[f] &= \{\,f^{\sigma} \mid \sigma\in \Sym(\ell)\,\} 
\end{align*}
of $f$ is called the \textit{orbit} of $f$ under $\Sym(\ell)$.

\begin{defn}[Hypergraph]
\label{def:hypergraph}
An \emph{$\ell$-uniform hypergraph} is a pair $\H=(V,E)$ of finite sets, 
where $V$ is the \emph{vertex set}, and  
$E$ is a set of orbits of $V^{\ell}$, called \emph{(hyper)edges}, 
under the natural action of $\Sym(\ell)$.
\end{defn} 

Our notational choices merit some explanation.
First, we allow repetitions of vertices in an edge, which is why we consider 
tuples $V^{\ell}$ rather than $\ell$-subsets ${V\choose \ell}$. 
Secondly, we consider \textit{undirected} hypergraphs, 
so edges are really \textit{multisets} rather than sequences
since the order of the vertices does not matter. 
In contrast to other multiset notations,
we account for symmetry using the action of $\Sym(\ell)$ 
on tuples because this choice aligns better with the invariants upon 
which our methods rely. In particular, 
by choosing a representative, we regard $e\in E$ 
as an element of $V^{\ell}$ and obtain other elements of the 
equivalence class as $e^{\sigma}$ for 
$\sigma\in\Sym(\ell)$.

\begin{defn}[Connectedness]
    \label{def:connected}
A \emph{walk} in an $\ell$-uniform hypergraph 
$\H$ is an alternating sequence $x_1,e_1,x_2,\ldots,x_{m-1},e_{m-1},x_m$
of vertices and edges, 
where for each $1\leq i\leq m-1$, there exist $a_i,b_i\in [\ell]$ such that 
$e_i(a_i)=x_{i}$ and $e_i(b_i)=x_{i+1}$. The \emph{reachability relation} 
on $V$, where $y$ is related to $x$ if there is a walk in $\H$ 
from $x$ to $y$, is an equivalence relation; 
the \emph{connected components} of $\H$ are the equivalence classes. 
We say that $\H$ is \emph{connected} if it has just one connected component.
\end{defn}

The final hypergraph notion we need is that of a quotient. As 
in the theory of graphs, the quotient of a hypergraph $\H$ provides 
a simplified view of $\H$. 

\begin{defn}[Quotients] 
    \label{def:quotient}
Let $\H=(V,E)$ be an $\ell$-uniform hypergraph, 
and let $\approx$ be an equivalence relation on $V$.
Define:
\begin{itemize}
\item $\bar{V}=\{\,\bar{x} \mid x\in V\,\}$, 
the set of $\approx$ equivalence classes; and 
\item $\bar{E} =\{ \bar{e} \mid e\in E\}\subseteq \bar{V}^{\ell}$, 
where $\bar{e}(a)=\bar{e(a)}$ for $a\in [\ell]$.
\end{itemize}
The $\ell$-uniform hypergraph $\bar{\H}=(\bar{V},\bar{E})$ 
is called the \emph{quotient} of $\H$ by $\approx$.
\end{defn}

\begin{ex}
    \label{ex:components}
Let $\H=(V,E)$ be any $\ell$-uniform hypergraph. We have seen that 
equivalence classes of $V$ under the reachability relation are the 
connected components of $\H$. Thus, in the quotient hypergraph 
$\bar{\H}=(\bar{V},\bar{E})$ for the reachability relation, 
$\bar{V}$ contains one vertex for each connected component.
For $e\in E$, there is a connected component $C\subseteq V$ such 
that $e(a)\in C$ for all $a\in [\ell]$, so it follows that $\bar{e}(a)=\bar{e(a)}=\bar{x}$, 
where $\bar{x}=C$. Thus, $\bar{E}$ contains exactly one edge $\bar{e}$ 
for each $\bar{x}\in\bar{V}$ with $\bar{e}(a)=\bar{x}$ for all $a\in [\ell]$. 
\end{ex}

Observe that $\bar{E}$ may contain multisets that are not sets 
even when $E\subseteq {V \choose \ell}$.

\begin{ex}
    \label{ex:5-vertex}
Let $\H=(V,E)$ be the (connected) 3-uniform hypergraph where $V=\{u,v,w,x,y\}$ and 
$E=\{\,(u,v,w)\,,\,(u,w,x)\,,\,(u,x,y)\,\}$, as depicted in Figure~\ref{fig:5-vertex}.

\begin{figure}[H]
    \centering
    \begin{tikzpicture}[scale = 0.8]
    \fill[red!20] (90:2) -- (162:2) -- (234:2) -- cycle;   
    \fill[green!20] (90:2) -- (234:2) -- (306:2) -- cycle;  
    \fill[blue!20] (90:2) -- (306:2) -- (18:2) -- cycle;    

    \node[circle, draw, fill=white] (1) at (90:2) {$u$};
    \node[circle, draw, fill=white] (2) at (162:2) {$v$};
    \node[circle, draw, fill=white] (3) at (234:2) {$w$};
    \node[circle, draw, fill=white] (4) at (306:2) {$x$};
    \node[circle, draw, fill=white] (5) at (18:2) {$y$};

   \draw[thin] (1) -- (2);  
   \draw[thin] (2) -- (3);  
   \draw[thin] (3) -- (4);  
   \draw[thin] (4) -- (5);  
   \draw[thin] (5) -- (1);  
   \draw[thin] (1) -- (3);  
   \draw[thin] (1) -- (4);  
    \end{tikzpicture}
    \caption{A hypergraph on 5 vertices with 3 hyperedges.}
    \label{fig:5-vertex}
\end{figure}
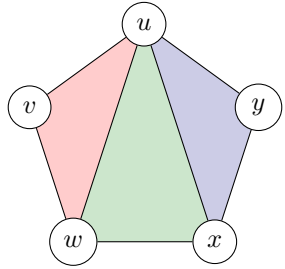

Figure~\ref{fig:quotients} shows quotients of $\H$ by 
two different equivalence relations.
Note, the generality of our edge notation means that quotients 
of $\ell$-uniform hypergraphs remain $\ell$-uniform hypergraphs.
The relation on the right is the one obtained from the 
automorphism group of the hypergraph~\cite[Section 5]{BS-Building-Blocks}.

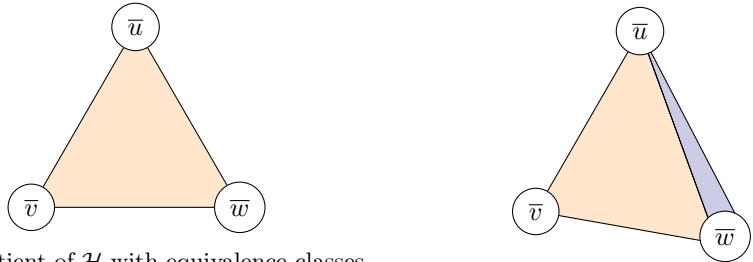
\begin{figure}[htbp]
    \centering

    \begin{subfigure}[b]{0.48\textwidth} 
        \centering
        \begin{tikzpicture}[scale=0.8]
        \fill[orange!20] (90:2) -- (210:2) -- (330:2) -- cycle;

        \node[circle, draw, fill=white, minimum size=0.5cm] (A) at (90:2) {$\bar{u}$};
        \node[circle, draw, fill=white, minimum size=0.5cm] (B) at (210:2) {$\bar{v}$};  
        \node[circle, draw, fill=white, minimum size=0.5cm] (C) at (330:2) {$\bar{w}$};  
    
        \draw[thin] (A) -- (B);
        \draw[thin] (B) -- (C);
        \draw[thin] (C) -- (A);
    \end{tikzpicture}
        \caption{Quotient of $\H$ with equivalence 
        classes $\bar{u}=\{u\}$, $\bar{v}=\{v,x\}$, $\bar{w}=\{w,y\}$.
        \newline ~}
        \label{fig:pictureA}
    \end{subfigure}
    \hfill 
    \begin{subfigure}[b]{0.48\textwidth} 
        \centering
        
        \begin{tikzpicture}[scale=0.8]
        \fill[orange!20] (90:2) -- (210:2) -- (310:2) -- cycle;
        \fill[blue!20] (90:2) -- (310:2) -- (325:2) -- cycle;

        \draw[thin] (90:2) -- (310:2) -- (325:2) -- cycle;
        \draw[thin] (90:2) -- (210:2) -- (310:2) -- cycle;
    
        \node[circle, draw, fill=white, minimum size=0.5cm] (A) at (90:2) {$\bar{u}$};
        \node[circle, draw, fill=white,  minimum size=0.5cm] (B) at (210:2) {$\bar{v}$};  
        \node[circle, draw, fill=white,  minimum size=0.5cm] (C) at (315:2) {$\bar{w}$};  
    \end{tikzpicture}
        \caption{Quotient of $\H$ with equivalence 
        classes $\bar{u}=\{u\}$, $\bar{v}=\{v,y\}$, $\bar{w}=\{w,x\}$. 
        }
        \label{fig:pictureB}
    \end{subfigure}

    \caption{Quotients of $\H$ by two different equivalence relations.}
    \label{fig:quotients}
\end{figure}
\end{ex}

\section{Signals}
\label{sec:derivations}
Since an edge $e$ of an $\ell$-uniform hypergraph 
$\H=(V,E)$ is an orbit $[e]=\{e^{\sigma}\mid \sigma\in\Sym(\ell)\}$, 
to each $\ell$-uniform hypergraph $\H=(V,E)$ one can associate an 
array $\Gamma:V^{\ell}\to \{0,1\}$
where, for $f\in V^{\ell}$,
\begin{align}
\label{eq:hypergraph-array}
\Gamma(f) &= \left\{ \begin{array}{ll} 1 & \mbox{if}\;[f]\in E, \\
    0 & \mbox{else}. \end{array} \right.
\end{align}
Each axis of $\Gamma$ (labelled by $a\in[\ell]$) has the same number of 
coordinates (indexed by $V$), so $\Gamma$ is a cubical array.
It is also symmetric because $\Gamma(f)=\Gamma(f^{\sigma})$ 
for all $f\in V^{\ell}$ and all $\sigma\in\Sym(\ell)$.
The $5\times 5\times 5$ symmetric array $\Gamma$ for the 5-vertex hypergraph in Example~\ref{ex:5-vertex} is 
pictured in Figure~\ref{fig:array}.

\begin{figure}[htb]
    \centering
    
    \begin{minipage}{0.48\textwidth}
        \centering
        \begin{tikzpicture}[scale = 1.0] 
            \fill[red!20] (90:2) -- (162:2) -- (234:2) -- cycle;   
            \fill[green!20] (90:2) -- (234:2) -- (306:2) -- cycle;  
            \fill[blue!20] (90:2) -- (306:2) -- (18:2) -- cycle;    
        
            \node[circle, draw, fill=white] (1) at (90:2) {$u$};
            \node[circle, draw, fill=white] (2) at (162:2) {$v$};
            \node[circle, draw, fill=white] (3) at (234:2) {$w$};
            \node[circle, draw, fill=white] (4) at (306:2) {$x$};
            \node[circle, draw, fill=white] (5) at (18:2) {$y$};
        
           \draw[thin] (1) -- (2);  
           \draw[thin] (2) -- (3);  
           \draw[thin] (3) -- (4);  
           \draw[thin] (4) -- (5);  
           \draw[thin] (5) -- (1);  
           \draw[thin] (1) -- (3);  
           \draw[thin] (1) -- (4);  
        \end{tikzpicture}
    \end{minipage}%
    \hfill
    \begin{minipage}{0.48\textwidth}
        \centering
        \def\N{5} 
        \newcommand{\drawcube}[4][black]{
          \pgfmathsetmacro{\X}{#2-1} 
          \pgfmathsetmacro{\Y}{#3-1} 
          \pgfmathsetmacro{\Z}{#4-1} 
            \coordinate (A) at (\X,\Y,\Z);
            \coordinate (B) at (\X+1,\Y,\Z);
            \coordinate (C) at (\X+1,\Y+1,\Z);
            \coordinate (D) at (\X,\Y+1,\Z);
            \coordinate (E) at (\X,\Y,\Z+1);
            \coordinate (F) at (\X+1,\Y,\Z+1);
            \coordinate (G) at (\X+1,\Y+1,\Z+1);
            \coordinate (H) at (\X,\Y+1,\Z+1);
            
            \filldraw[draw=black!70, fill=#1!15, opacity=0.8] (A) -- (B) -- (F) -- (E) -- cycle; 
            \filldraw[draw=black!70, fill=#1!18, opacity=0.8] (A) -- (D) -- (H) -- (E) -- cycle; 
            \filldraw[draw=black!70, fill=#1!12, opacity=0.8] (A) -- (B) -- (C) -- (D) -- cycle; 
            \filldraw[draw=black!70, fill=#1!22, opacity=0.8] (B) -- (C) -- (G) -- (F) -- cycle; 
            \filldraw[draw=black!70, fill=#1!25, opacity=0.8] (D) -- (C) -- (G) -- (H) -- cycle; 
            \filldraw[draw=black!70, fill=#1!30, opacity=0.8] (E) -- (F) -- (G) -- (H) -- cycle; 
        }

        \tdplotsetmaincoords{70}{110} 
        
        \begin{tikzpicture}[tdplot_main_coords, scale=0.7, 
            line join=round, line cap=round,
            outerframe/.style={gray!70, very thick},
            hiddenframe/.style={gray!50, thin, dotted} 
        ]

            \def\u{1}\def\v{2}\def\w{3}\def\x{4}\def\y{5}
            \foreach \a/\b/\c in {1/\x/\y, 1/\y/\x, \x/1/\y, \x/\y/1, \y/1/\x, \y/\x/1}{\drawcube[blue!70!black]{\a}{\b}{\c}}
            \foreach \a/\b/\c in {1/\x/\w, 1/\w/\x, \x/1/\w, \x/\w/1, \w/1/\x, \w/\x/1}{\drawcube[green!70!black]{\a}{\b}{\c}}
            \foreach \a/\b/\c in {1/\v/\w, 1/\w/\v, \v/1/\w, \v/\w/1, \w/1/\v, \w/\v/1}{\drawcube[red]{\a}{\b}{\c}}
        
            \coordinate (C000) at (0,0,0); \coordinate (CN00) at (\N,0,0);
            \coordinate (CNN0) at (\N,\N,0); \coordinate (C0N0) at (0,\N,0);
            \coordinate (C00N) at (0,0,\N); \coordinate (CN0N) at (\N,0,\N);
            \coordinate (CNNN) at (\N,\N,\N); \coordinate (C0NN) at (0,\N,\N);


        \end{tikzpicture}
    \end{minipage}
    
    \caption{On the left is a 3-uniform hypergaph $\H$ on 5 vertices. On the 
    right is the $5\times 5\times 5$ array $\Gamma$ corresponding to $\H$, 
    where the $3!=6$ nonzero entries for each colored edge are represented as blocks 
    of the same color. This helps to see the symmetry of $\Gamma$.}
    \label{fig:array}
\end{figure}
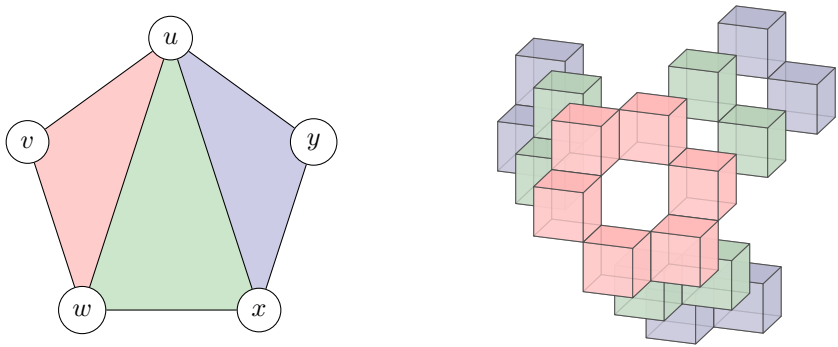

In this paper we work directly with the hypergraph $\H$ rather than with 
the array $\Gamma(\H)$. 
The reason for mentioning $\Gamma(\H)$ at all is that arrays are representations 
of tensors, and it is from tensors that we get the invariants 
we shall use throughout the paper. Indeed, the following key 
definition is an adaption to hypergraphs of the 
notion of a \textit{derivation} for tensors~\cite[Definition 5.1]{Strata}.

\begin{rem}
\label{rem:infinite-field}
Throughout the paper, $\K$ denotes an arbitrary field.
The specific role of $\K$ is inconsequential except that we shall
eventually need it to contain enough distinct elements,
so for convenience we henceforth assume that $\K$ 
is infinite. 
\end{rem}

\begin{defn}[$T$-Signal]
    \label{def:derivation}
Given an $\ell$-uniform hypergraph $\H=(V,E)$ and a linear 
transformation $T:\K^{\ell}\to\K^r$, where $r$ is a positive integer, a 
\emph{$T$-signal} of $\H$
is a function $\delta:[\ell]\to \K^V$ sending $a\mapsto (\delta_a:V\to \K)$ 
satisfying 
\begin{align}
\label{eq:derivation}
\forall e\in E, & ~\forall \sigma\in \Sym(\ell), & T(\delta * e^{\sigma})=0\in\K^r,
\end{align}
where $\delta * e^{\sigma}:[\ell]\to \K$ maps $a\mapsto \delta_a(e^{\sigma}(a))$. Put
\begin{align}
\label{eq:der-space}
    \Der(\H,T) &:= \{\;\delta\in (\K^V)^{\ell}\mid \delta~\mbox{is a $T$-signal of $\H$}\;\}.
\end{align}
\end{defn}

Let $\mathbf{1}:V\to\K$ be the constant function $\mathbf{1}(x)=1$ for 
$x\in V$, and let $\mathbf{1}_a$ denote the function $\mathbf{1}$ 
operating on coordinate $a\in[\ell]$. Similarly, let $\mathbf{0}:V\to\K$ 
be the zero function $\mathbf{0}(x)=0$ for $x\in V$.
Each $T:\K^{\ell}\to\K^r$ has a kernel 
$\ker(T)=\{\lambda:[\ell]\to\K\mid T(\lambda)=0\}$. 
Observe for $\lambda\in\ker(T)$ that the function  
\begin{align}
    \delta(\lambda):[\ell]\to \K^V, &&  \delta(\lambda)_a=\lambda(a)\mathbf{1}_a
\end{align} 
is a $T$-signal of any hypergaph $\H$.
Such signals are constant on every axis, so we call them
\textit{constant signals} and put 
\begin{align}
\label{eq:constant-der}
\Der_0(T) &= \{\;\delta(\lambda)\mid \lambda\in\ker(T)\;\}.
\end{align}
Constant signals clearly tell us nothing useful about the structure 
of a hypergraph.
For $a\in [\ell]$, define $\hat{a}:\K^{\ell}\to\K$ where 
$\hat{a}(b)=1$ if $a=b$ and $\hat{a}(b)=0$ otherwise, 
so $\{\,\hat{a}\mid a\in [\ell]\,\}$ is the standard basis of $\K^{\ell}$.
Similarly, let $\{\,\hat{x}\mid x\in V\,\}$ be the standard basis of $\K^V$, 
where $\hat{x}(y)=1$ if $x=y$ and $\hat{x}(y)=0$ otherwise.

\begin{ex}
    \label{ex:U-signal}
    Consider the `triangle' hypergraph $\H$ on vertex set $V=\{u,v,w\}$ 
    with the single hyperedge $e=(u,v,w)$. Let $T:\K^3\to \K$ be the linear 
    transformation defined by $T(\lambda)=\lambda(1)-2\lambda(2)+\lambda(3)$.  
    Then $\ker(T)$ is spanned by the vectors
    $2\cdot \hat{1}+\hat{2}$ and $\hat{1}-\hat{3}$, so 
    $\Der_0(T)$ contains the (constant) signals
    \begin{align*}
    \delta(2\hat{1}+\hat{2})=(2\cdot\mathbf{1}_1,\mathbf{1}_2,\mathbf{0}_3), &&
    \delta(\hat{1}-\hat{3})=(\mathbf{1}_1,\mathbf{0}_2,-\mathbf{1}_3).
    \end{align*}
    One can easily verify that 
    $\delta:[3]\to \K^V$ defined by
    $\delta_1=2\hat{w}$, $\delta_2=\hat{u}+\hat{v}$, $\delta_3=2\hat{w}$
    is a (non-constant) $T$-signal.
    For example, if $\sigma\in\Sym(\ell)$ is the 3-cycle $(1,2,3)$, then
    \begin{align*}
        T(\delta*e^{\sigma}) &= \delta_1(e(1^{\sigma})) - 2\delta_2(e(2^{\sigma})) + \delta_3(e(3^{\sigma})) \\
        &= \delta_1(e(2)) - 2\delta_2(e(3)) + \delta_3(e(1)) \\
        &= 2\hat{w}(v)-2(\hat{u}(w)+\hat{v}(w))+2\hat{w}(u) \\
        &= 0 - 2(0+0) + 0 = 0.
    \end{align*}
\end{ex}

The next result establishes the linear nature of the invariants $\Der(\H,T)$.

\begin{lemma}
    \label{lem:der-space}
    Let $T:\K^{\ell}\to \K^r$, and let $\H=(V,E)$ be an $\ell$-uniform hypergraph. Then
    $\Der(\H,T)$ is a subspace of $(\K^V)^{\ell}$, 
    and $\Der_0(T)$ is a subspace of $\Der(\H,T)$.
\end{lemma}

\begin{proof}
Evidently, $(\K^V)^{\ell}$ is a $\K$-vector space under 
addition $(\delta+\acute{\delta})_a:=\delta_a+\acute{\delta}_a$
and scalar product $(k\delta)_a:=k\delta_a$.
For $\delta,\acute{\delta}\in (\K^V)^{\ell}$,  
$e\in E$, $\sigma\in\Sym(\ell)$, $k\in \K$, $a\in [\ell]$,
\begin{align*}
[(\delta+k\acute{\delta})*e^{\sigma}](a) &=(\delta+k\acute{\delta})_a(e^{\sigma}(a))
& \mbox{(definition of}\;\delta*e^{\sigma})\\
&=(\delta_a+k\acute{\delta}_a)(e^{\sigma}(a)) & \mbox{(operations on $(\K^V)^{\ell}$)}\\
&=\delta_a(e^{\sigma}(a))+k\acute{\delta}_a(e^{\sigma}(a)) & \mbox{(operations on $\K^V$)}\\
&=(\delta*e^{\sigma})(a)+k(\acute{\delta}*e^{\sigma})(a) & \mbox{(definition of}\;\delta*e^{\sigma}).
\end{align*}
Hence, $(\delta+k\acute{\delta})*e^{\sigma}= 
(\delta*e^{\sigma})+k(\acute{\delta}*e^{\sigma})\in \K^{\ell}$.
Since $T:\K^{\ell}\to \K^r$ is a $\K$-linear transformation, 
for $\delta,\acute{\delta}\in \Der(\H,T)$
we have 
\begin{align*}
T((\delta+k\acute{\delta})*e^{\sigma}) &=
T((\delta*e^{\sigma})+k(\acute{\delta}*e^{\sigma})) \\
&=T(\delta*e^{\sigma}) + k\,T(\acute{\delta}*e^{\sigma})=0.
\end{align*}
It follows that $\delta+k\acute{\delta}\in\Der(\H,T)$, so $\Der(\H,T)$ 
is a subspace of $(\K^V)^{\ell}$.
Finally, let $\lambda\in\ker(T)\subseteq \K^{\ell}$. 
For $e\in E$, $\sigma\in\Sym(\ell)$,
we have
\begin{align*}
  \forall a\in [\ell], &&  \delta(\lambda)*e^{\sigma}(a) &= \delta(\lambda)_a(e^{\sigma}(a))=\lambda(a),
\end{align*}
so $\delta(\lambda)*e^{\sigma}=\lambda$. Hence, $T(\delta(\lambda)*e^{\sigma})=T(\lambda)=0$, 
so $\delta(\lambda)\in\Der(\H,T)$. The map $\lambda\mapsto \delta(\lambda)$ 
is a $\K$-linear embedding of $\ker(T)$ into $\Der(\H,T)$.
\end{proof}

\begin{rem}[Signals versus derivations]
    \label{rem:comb-v-gen}
The notion of $T$-derivation in~\cite[Definition 5.1]{Strata} can be 
applied directly to the \textit{array} $\Gamma$ of a hypergraph $\H$, in 
which case each $\delta_a$ would be a linear operator on the vector 
space $\K^V$. The notion of signal in Definition~\ref{def:derivation}, 
where each $\delta_a$ is a function $V\to \K$, is 
equivalent to requiring that $\delta_a$ is a \textit{diagonal} operator.
Thus, our derivations only \textit{scale} vertices of $\H$, 
whereas the derivations in~\cite{Strata} involve arbitary linear combinations.
In this sense we consider our hypergraph signals to be 
`combinatorial' analogues of tensor derivations.
\end{rem}

\begin{rem}[Computing signals]
    \label{rem:computing-der}
Using signals rather than derivations, we can push the range 
of practical applications of our linear invariants of hypergraphs.
The unknown $\delta\in\Der(\H,T)$ can be written as 
linear combinations of a standard basis $(\varepsilon^{(a,x)}\mid a\in [\ell],x\in V)$
of $(\K^V)^{\ell}$ where, if 
$\varepsilon=\varepsilon^{(a,x)}$, then $\varepsilon_b(y)=1$ if and $a=b$ and $x=y$, 
and $\varepsilon_b(y)=0$ otherwise. Thus, we can represent 
$\delta\in \Der(\H,T)$ as a column vector 
$(\delta_{a}(x)\mid a\in [\ell],x\in V)$ of degree $\ell\cdot|V|$.

Similarly, $T$ is represented 
as a $r\times\ell$ matrix $M=(M_{i,a}\mid i\in [r],\,a\in [\ell])$.
Each $i\in [r]$, $e\in E$, $\sigma\in \Sym(\ell)$ 
in condition~\eqref{eq:derivation}
contributes one linear equation in the 
unknowns $\delta_a(x)$. Indeed, $\Der(\H,T)$ is precisely 
the nullspace of a matrix with $r\cdot|E|\cdot\ell!$ rows 
and $\ell\cdot |V|$ columns,
whose entry in the row labelled by $(i,e,\sigma)$ and 
the column labelled by $(a,x)$ is $M_{i,a}$.

For many applications the uniformity constant $\ell$ 
of the hypergraph $\H$ will be quite small (say 3 or 4), 
and $r$ will typically be smaller than $\ell$. In such cases, 
if $\H$ has $n$ vertices and $m$ edges, we can solve for $\Der(\H,T)$ 
using $O(mn\cdot\min\{m,n\})$ operations.
This is much better than the complexity for 
derivations in~\cite{Strata}.
\end{rem}

\section{Centroids}
\label{sec:centroids}
We focus now on a particular transformation $C$ and explore 
the $C$-signals of a hypergraph. 
Define $C:\K^{\ell}\to\K^{\ell-1}$ to be the map $\lambda\mapsto C(\lambda)$, 
where
\begin{align}
\label{eq:centroid}
C(\lambda)(i) &= \lambda(i)-\lambda(i+1), & \mbox{for}~ i\in [\ell-1]
\end{align}
Note that $\lambda\in \ker(C)$ 
if, and only if, $\lambda(1)=\cdots=\lambda(\ell)=k$ for some $k\in \K$,
so $\Der_0(C) =\{\, k\mathbf{1}_a\mid a\in [\ell],\,k\in \K  ~\}$
is 1-dimensional.
\smallskip

For general tensors, the transformation $C$ parameterizes null patterns in 
which the nonzero entries cluster as blocks down the leading 
diagonal~\cite[Section 7.3]{Strata}. Similarly, 
$C$-signals detect the connected components of a hypergraph.

\begin{lem}
     \label{lem:centroid-constant}
     Let $\H=(V,E)$ be an $\ell$-uniform hypergraph, and
     $\delta=(\delta_a\mid a\in [\ell])$ a $C$-signal of $\H$. 
     If $x,y\in V$ are in the same connected component of $\H$, then 
     \begin{align}
     \label{eq:centroid-constant}
     \forall a,b\in [\ell],&& \delta_a(x)=\delta_b(y).
     \end{align}
\end{lem}

\begin{proof}
Let $\delta$ be a $C$-signal of $\H$. 
For each $e\in E$ and $\sigma\in\Sym(\ell)$, we have
$\delta*e^{\sigma}\in \ker(C)$, 
so $\delta_1(e^{\sigma}(1))=\cdots=\delta_{\ell}(e^{\sigma}(\ell))$.

Suppose first that $x$ and $y$ belong to some edge $e\in E$.
Then, for each distinct pair $a,b\in [\ell]$, there exists 
$\sigma\in\Sym(\ell)$ such that $e^{\sigma}(a)=e(a^{\sigma})=x$ and 
$e^{\sigma}(b)=y$, so it follows that $\delta_a(x)=\delta_b(y)$.
If $\tau$ is the transposition of $\Sym(\ell)$ interchanging $a$ 
and $b$, then $e^{\tau\sigma}(a)=e^{\sigma}(b)=y$, $e^{\tau\sigma}(b)=e^{\sigma}(a)=x$, 
and $e^{\tau\sigma}(c)=e^{\sigma}(c)$ for all $c\in [\ell]-\{a,b\}$.
Since $\ell\geq 3$, there exists $c\in [\ell]-\{a,b\}$, so we have 
\begin{align*}
\delta_a(x)&=\delta_a(e^{\sigma}(a))=\delta_c(e^{\sigma}(c))=
\delta_c(e^{\tau\sigma}(c))=\delta_a(e^{\tau\sigma}(a))=\delta_a(y).
\end{align*}
The lemma now follows by induction on the length of a
walk in $\H$ from $x$ to $y$.
\end{proof}

\begin{rem}
    \label{rem:valence}
The bound $\ell\geq 3$ is important to note. The tensor  
invariants upon which our hypergraph invariants are based
are defined only when the \textit{valence} of the tensor 
is greater than 2. They are inherently \textit{multi}-linear rather 
than linear invariants, and in particular cannot be applied to graphs.
Indeed the $\ell\geq 3$ assumption was needed at the end of the proof 
of Lemma~\ref{lem:centroid-constant}.
\end{rem}

\begin{thm}
    \label{thm:centroid-connected}
The number of connected components of an $\ell$-uniform hypergraph 
$\H$ is equal to the dimension of $\Der(\H,C)$.
\end{thm}

\begin{proof}
Let $V_1,\ldots,V_d$ be the connected components 
of $\H$. We exhibit a basis of $\Der(\H,C)$ 
of size $d$. 
For $i\in \{1,\ldots,d\}$, define $\mathbf{1}^{[i]}\in(\K^V)^{\ell}$, 
where for $x\in V$,
\begin{align}
\label{eq:der-i}
\forall a\in [\ell], &&
\mathbf{1}^{[i]}_a(x) &:=  \left\{ \begin{array}{ll} 1 & x\in V_i \\ 0 & \mbox{else} \end{array} \right..
\end{align}
For $e\in E$ and $\sigma\in\Sym(\ell)$, since $\mathbf{1}^{[i]}_a(e^{\sigma}(a))=1$
for $a\in[\ell]$ 
precisely when $e$ contains vertices in $V_i$, it follows 
that $C(\mathbf{1}^{[i]}*e^{\sigma})=0$, and
hence that $\mathbf{1}^{[i]}$ is a $C$-signal.

For scalars $k_1,\ldots,k_d\in \K$ and $x\in V_i$, we have 
\begin{align*}
\forall a\in [\ell],
&&
(k_1\mathbf{1}^{[1]}+\ldots+k_d\mathbf{1}^{[d]})_a(x) = k_i.
\end{align*}
Hence, $k_1\mathbf{1}^{[1]}+\ldots+k_d\mathbf{1}^{[d]}=\mathbf{0}$ if, and only if, 
$k_1=\ldots=k_d=0$, so the signals $\mathbf{1}^{[i]}$ 
are linearly independent. Finally, if $\delta$ is a $C$-signal of $\H$,
by Lemma~\ref{lem:centroid-constant} there exist scalars $k_1,\ldots,k_d\in \K$  
such that $\delta=k_1{}\mathbf{1}^{[1]}+\ldots+k_d\mathbf{1}^{[d]}$. It follows that 
$\{\,\mathbf{1}^{[i]}\mid i\in\{1,\ldots,d\}\,\}$ is a basis for 
$\Der(\H,C)$.
\end{proof}

\section{Universal signals}
\label{sec:universal}
In view of Theorem~\ref{thm:centroid-connected}, we henceforth work 
just with connected hypergraphs: 
\begin{align}
\label{eq:define-hypl}
\mathtt{HYP}_{\ell} &=~\{~\H \;\mid\; \H~\mbox{is a \underline{connected} $\ell$-uniform hypergraph}~\}.
\end{align}
The previous section focused on signals of the 
transformation $C$. In this section we consider another 
transformation $U:\K^{\ell}\to\K$, where 
\begin{align}
\label{eq:universal}
U(\lambda) &= \lambda(1)+\lambda(2)+\cdots+\lambda(\ell).
\end{align}
For each $a\in[\ell-1]$, define $\lambda^{(a)}\in\K^{\ell}$, where
\begin{align*}
\lambda^{(a)}(b) &= \left\{ \begin{array}{rl} 1 & \mbox{if}~b=a \\ -1 & \mbox{if}~b=a+1 \\ 0 & \mbox{else} \end{array} \right..  
\end{align*}
Then $\{\,\lambda^{(a)}\mid a\in[\ell-1]\,\}$ is a basis for $\ker(U)$, so 
$\{\,\delta(\lambda^{(a)})\mid a\in[\ell-1]\,\}$ is a basis for 
$\Der_0(U)$, which therefore has dimension $\ell-1$.

\begin{rem}
    \label{rem:engaged}
    If $T(\hat{a})=0\in \K^r$ for some $a\in [\ell]$, then
    condition~\eqref{eq:derivation} imposes no constraint on 
    $\delta_a$. Although such $T$ are used to 
    decompose tensors relative to a fixed 
    set of axes~\cite[Section 7.2]{Strata}, they provide no 
    useful information for the symmetric tensors  
    underlying hypergraphs. We therefore 
    assume that $T(\hat{a})\neq 0$ for all $a\in [\ell]$, 
    which is evidently the case when $T=U$.
\end{rem}

The following result is a special case of~\cite[Proposition 8.1]{Strata}, 
but we include a proof tailored to our notation.
It tells us that, for any $T$, the $\K$-space of $T$-signals 
of a hypergraph $\H$ embeds into its space of $U$-signals.

\begin{prop}
\label{prop:universal}
Let $\H\in\mathtt{HYP}_{\ell}$, and let $T:\K^{\ell}\to \K^r$ 
be a linear map. 
There are nonzero scalars $k_1,\ldots,k_{\ell}$ such that 
\begin{align*}
(\delta_1,\ldots,\delta_{\ell}) \in\Der(\H,T) & \qquad \implies \qquad 
(k_1\delta_1,\ldots,k_{\ell}\delta_{\ell})\in\Der(\H,U).
\end{align*}
\end{prop}

\begin{proof}
Let $T:\K^{\ell}\to\K^r$ with $T(\hat{a})\neq 0$ for all $a\in [\ell]$ (cf.~Remark~\ref{rem:engaged}).
Then $\{v\in \K^r\mid v\cdot T(\hat{a})=0\}$, 
where $\cdot$ is the usual dot product,
is a subspace of $\K^r$ of codimension 1.
Since $\K$ is infinite, $\K^r$ is not a finite union of hyperplanes, so
there exists $w\in \K^r$ such that 
\begin{align*}
\forall a\in [\ell], && 
w\cdot T(\hat{a}) \neq 0. 
\end{align*}
Define $T_w:\K^{\ell}\to \K$, where $T_w(u)=w\cdot T(u)\in \K$. 
Since $(\delta_1,\ldots,\delta_{\ell})$ is a $T$-signal, 
it is clear from~\eqref{eq:derivation} it 
is also a $T_w$-signal. Since $T_w(\hat{a})=w\cdot T(\hat{a})\neq 0$, 
putting $k_a=(w\cdot T(\hat{a}))^{-1}$,
it follows that $(k_1\delta_1,\ldots,k_{\ell}\delta_{\ell})$ 
is a $U$-signal, as required.
\end{proof}

Proposition~\ref{prop:universal} tells us that, under very mild conditions, 
the transformation $U$ is universal among the 
various choices of $T$. We conclude this section with a useful analogue of 
Lemma~\ref{lem:centroid-constant}.

\begin{lem}
     \label{lem:derivation-constant}
     For any $\H\in\mathtt{HYP}_{\ell}$ and
     $\delta\in\Der(\H,U)$, we have
     \begin{align*}
     \forall a,b\in [\ell], && \delta_a-\delta_b \in \langle \mathbf{1} \rangle.
     \end{align*}
\end{lem}

\begin{proof}
The claim holds if $a=b$, so let $a$ and $b$ be distinct.
If  $\delta\in\Der(\H,U)$, then 
\begin{align*}
    \forall e\in E,~\forall \sigma\in\Sym(\ell), &&
    U(\delta*e) = \sum_{c=1}^{\ell}\delta_c(e^{\sigma}(c))=0.
\end{align*}
We must show, for distinct vertices $x,y\in V$,
that $\delta_a(x)-\delta_b(x)=\delta_a(y)-\delta_b(y)$. 

Suppose first that there is an edge $e\in E$
containing $x$ and $y$. Then there exists $\sigma\in\Sym(\ell)$ such that 
$e^{\sigma}(a)=x$ and $e^{\sigma}(b)=y$. Since $\delta$ is a 
$U$-signal of $\H$, 
\begin{align}
    \label{eq:der-sigma}
0 
~ = ~ \sum_{c\in [\ell]}\delta_c(e^{\sigma}(c))
~ = ~ \delta_a(x)+\delta_b(y)+\sum_{c\not\in\{a,b\}}\delta_c(e^{\sigma}(c)).
\end{align}
Let $\tau$ be the transposition of $\Sym(\ell)$ that interchanges 
$a$ and $b$. Then 
\[
e^{\tau\sigma}(a)=e^{\sigma}(b)=y, \quad e^{\tau\sigma}(b)=e^{\sigma}(a)=x, \quad 
e^{\tau\sigma}(c)=e^{\sigma}(c)~\mbox{for}~c\not\in\{a,b\}, 
\]
so we have 
\begin{align}
    \label{eq:der-tau-sigma}
0 
~ = ~ \sum_{c\in [\ell]}\delta_c(e^{\tau\sigma}(c))
~ = ~ \delta_a(y)+\delta_b(x)+\sum_{c\not\in\{a,b\}}\delta_c(e^{\sigma}(c)).
\end{align}
Subtracting~\eqref{eq:der-tau-sigma} from~\eqref{eq:der-sigma} we get 
$\delta_a(x)-\delta_b(x)=\delta_a(y)-\delta_b(y)$.
\smallskip

The result now follows by induction on the length of a path in $\H$ 
from $x$ to $y$.
\end{proof}

\section{Fusion and frames}
\label{sec:der-quot}
We return now to a notion 
introduced in Section~\ref{sec:background},
namely quotients of hypergaphs. Given a connected $\ell$-uniform 
hypergraph $\H=(V,E)$, one can of course form a quotient 
$\bar{\H}=(\bar{V},\bar{E})$ using any 
equivalence relation on $V$, but doing so 
without reference to $\H$ will  
tell us nothing useful about its structure. Our strategy is to use the
$T$-signals of $\H$ to define an equivalence relation on $V$.

\begin{defn}[Fusion and Frame]
    \label{def:fusion}
Let $\H=(V,E)$ be a connected $\ell$-uniform hypergraph,
and let $T:\K^{\ell}\to\K^r$ be a linear map. The equivalence relation 
\begin{align}
\label{eq:C-equiv}
x \equiv_T y  \qquad \iff \qquad \forall \delta \in \Der(\H,T), \quad 
\delta(x) = \delta(y)
\end{align}
on $V$ is called the \emph{$T$-fusion} relation.
The quotient hypergraph $\mathfrak{F}(\H,T)$ of $\H$ under 
$T$-fusion is called the \emph{$T$-frame} of $\H$.
\end{defn}


The matter of which $T$ to choose 
was effectively settled by Proposition~\ref{prop:universal}. For 
any $T$ and any $\H=(V,E)$, the $\K$-space $\Der(\H,T)$ 
embeds in $\Der(\H,U)$. It follows immediately that $U$-fusion on 
$V$ is always a refinement of $T$-fusion. Thus, we will detect the 
most subtle features of $\H$ using $U$-fusion and computing 
the corresponding frame $\mathfrak{F}(\H):=\mathfrak{F}(\H,U)$.


\begin{ex}
    \label{ex:5-vertex-revisited}
    We revisit Example~\ref{ex:5-vertex} 
    and compute the $U$-frame of the now familiar 5-vertex hypergraph. 
    A calculation using our Python 
    implementation~\cite{Python} shows that $\Der(\H,U)$ 
    is 4-dimensional. 
    The subspace $\Der_0(U)$ is 2-dimensional, so 
    $\Der(\H,U)/\Der_0(U)$ is also 2-dimensional. 
    By Lemma~\ref{lem:derivation-constant} we need only consider 
    the projection of a signal onto the first coordinate.
    The space $\Der(\H,U)/\Der_0(U)$
    is spanned by the (images of the) 
    $U$-signals $\hat{u}$ and $\hat{w}+\hat{y}$.
    Thus, all $U$-signals 
    of $\H$ are constant on the partition $\{\,\{u\}\,,\,\{v,x\}\,,\,\{w,y\}\,\}$, 
    so there are three $U$-fusion classes: 
    $\bar{u}=\{u\}$, $\bar{v}=\{v,x\}$, $\bar{w}=\{w,y\}$.
    Thus, $\mathfrak{F}(\H)$ 
    consists of three vertices $\bar{u}\,,\,\bar{v}\,,\,\bar{w}$ 
    and a single hyperedge containing all three vertices, as 
    depicted in Figure~\ref{fig:5-vertex-der-quotient}.

    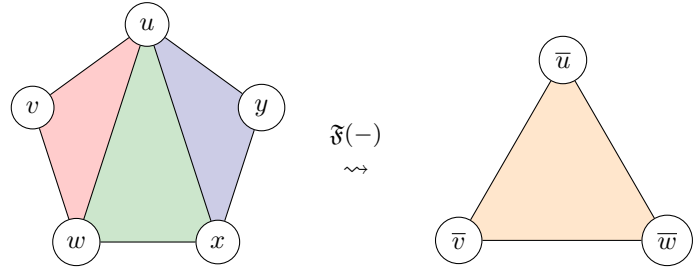
\begin{figure}[H]
    \centering

    \begin{subfigure}{0.4\textwidth} 
        \centering
        \begin{tikzpicture}[scale = 0.8]
    \fill[red!20] (90:2) -- (162:2) -- (234:2) -- cycle;   
    \fill[green!20] (90:2) -- (234:2) -- (306:2) -- cycle;  
    \fill[blue!20] (90:2) -- (306:2) -- (18:2) -- cycle;    

    \node[circle, draw, fill=white] (1) at (90:2) {$u$};
    \node[circle, draw, fill=white] (2) at (162:2) {$v$};
    \node[circle, draw, fill=white] (3) at (234:2) {$w$};
    \node[circle, draw, fill=white] (4) at (306:2) {$x$};
    \node[circle, draw, fill=white] (5) at (18:2) {$y$};

   \draw[thin] (1) -- (2);  
   \draw[thin] (2) -- (3);  
   \draw[thin] (3) -- (4);  
   \draw[thin] (4) -- (5);  
   \draw[thin] (5) -- (1);  
   \draw[thin] (1) -- (3);  
   \draw[thin] (1) -- (4);  
    \end{tikzpicture}
    \end{subfigure}
    \raisebox{0.5in}[0pt][0pt]{
    \begin{tikzpicture}[remember picture, overlay, baseline=(current bounding box.center)]
        \node (arrow_center) at (0, 0) {$\rightsquigarrow$};
        \node[above=0pt of arrow_center, font=\small] {$\mathfrak{F}(-)$};
    \end{tikzpicture}
    }
    \begin{subfigure}{0.4\textwidth} 
        \centering
        
        \begin{tikzpicture}[scale=0.8]
        \fill[orange!20] (90:2) -- (210:2) -- (330:2) -- cycle;

        \node[circle, draw, fill=white, minimum size=0.5cm] (A) at (90:2) {$\bar{u}$};
        \node[circle, draw, fill=white, minimum size=0.5cm] (B) at (210:2) {$\bar{v}$};  
        \node[circle, draw, fill=white, minimum size=0.5cm] (C) at (330:2) {$\bar{w}$};  
    
        \draw[thin] (A) -- (B);
        \draw[thin] (B) -- (C);
        \draw[thin] (C) -- (A);
    \end{tikzpicture}
    \end{subfigure}

    \caption{On the left is our hypergraph $\H$ on 5 vertices with 3 hyperedges, and 
    on the right is its derivation quotient $\D(\H)$.}
    \label{fig:5-vertex-der-quotient}
\end{figure}
\end{ex}

Consider the function $\D:\mathtt{HYP}_{\ell}\to \mathtt{HYP}_{\ell}$ that 
sends a hypergraph $\H$ to its frame $\D(\H)$.
In the rest of the paper we explore the properties of $\D(-)$, and 
present and analyze an algorithm to compute images under this function.
The next result helps with both objectives.
It provides a constructive proof that there 
is always a single ``generating" signal $\delta$ whose values determine the 
fusion classes.

\begin{prop}
    \label{prop:single-der} 
There is a polynomial-time algorithm that,
given a connected $\ell$-uniform hypergraph $\H=(V,E)$, finds
$\delta\in \Der(\H,U)$ satisfying
\begin{align}
    \label{eq:single-der}
    \forall x,y\in V, &&
x\equiv_U y ~~ \iff ~~ \delta(x) = \delta(y).
\end{align}
\end{prop}

\begin{proof}
We prove that Algorithm~\ref{algo:single-der} behaves as stated.
\smallskip

Given $\H=(V,E)$ with $|V|=n$ and $|E|=m$, a basis for 
$\Der(\H,U)$ can be computed
using $O(mn\cdot\min\{m,n\})$ operations (Remark~\ref{rem:computing-der}).
We suppress the details of this computation and initialize 
a basis for $\Der(\H,U)$ on Line 1.
\smallskip

    \begin{algorithm}    
    \caption{(Generating Signal)}
    \begin{algorithmic}[1]
    \Require $\H=(V,E)\in \mathtt{HYP}_{\ell}$.
    \Ensure $\delta\in \Der(\H,U)$ such that 
    $x\equiv_U y$ for $x,y\in V$ if, and only if, $\delta(x)=\delta(y)$.
    \State $\mathcal{B}\gets$ basis for $\Der(\H,U)$
    \State $\mathcal{C} \gets \{~\}$
    \State $\delta \gets 0$
    \While{$\exists \beta\in \mathcal{B}- \mathcal{C}$}
        \State $\mathcal{C} \gets \mathcal{C}\cup\{\beta\}$
        \State $k\gets $ any scalar in 
        $\K- \left\{\frac{\delta(x)-\delta(y)}{\beta(y)-\beta(x)}\mid 
        x,y\in V,\,\beta(x)\neq\beta(y)\,\right\}$ 
        \State $\delta \gets \delta + k\beta$
    \EndWhile
    \State\Return $\delta$
    \end{algorithmic}
    \label{algo:single-der}
    \end{algorithm}

    Observe that, trivially, the following predicate is true 
    entering the first iteration of the \textbf{while} loop on Line 4:
    \begin{align*}
    \mathtt{P}(\mathcal{C},\delta)\,: && 
    \delta\in\Der(\H,U) ~ \wedge ~
    (~(\forall x)(\forall y)[~
    \delta(x)=\delta(y) ~ \iff ~ \forall \alpha\in \mathcal{C},\,\alpha(x)=\alpha(y)
    ~]~)
    \end{align*}
    Consider a fixed iteration of the loop, where the pair $(\delta,\mathcal{C})$ has values 
    $(\delta_0,\mathcal{C}_0)$ before the iteration and $(\delta_1,\mathcal{C}_1)$ at the end,
    so we have $\delta_1=\delta_0+k\beta$ and $\mathcal{C}_1=\mathcal{C}_0\cup\{\beta\}$. 
    \smallskip
    
    Suppose that $\mathtt{P}(\mathcal{C}_0,\delta_0)$ holds, and fix $x,y\in V$.
    \smallskip

    If $\delta_0(x)=\delta_0(y)$ and $\beta(x)=\beta(y)$, then 
    $\alpha(x)=\alpha(y)$ for all $\alpha\in \mathcal{C}_1=\mathcal{C}_0\cup\{\beta\}$.
    But it also follows that 
    \begin{align*}
    \delta_1(x) &= \delta_0(x) + k\beta(x) = \delta_0(y)+k\beta(y) = \delta_1(y).
    \end{align*}
    Conversely, if $\delta_1(x)=\delta_1(y)$, then
    \begin{align*}
    \delta_0(x) + k\beta(x)=\delta_1(x) &= \delta_1(y) = \delta_0(y) + k\beta(y), 
    \end{align*}
    so that $\delta_0(x)-\delta_0(y)=k[\beta(y)-\beta(x)]$.
    If $\beta(x)\neq\beta(y)$, then 
    $k = \frac{\delta_0(x)-\delta_0(y)}{\beta(y)-\beta(x)}$, but we chose $k$ in Line 6 
    so that this does not occur. 
    Hence, $\beta(x) = \beta(y)$ and $\delta_0(x)=\delta_0(y)$, 
    so $\alpha(x)=\alpha(y)$ for all $\alpha\in \mathcal{C}_1=\mathcal{C}_0\cup\{\beta\}$.
    \smallskip

    Finally, since $\delta_0\in\Der(\H,U)$ and $\delta_1=\delta_0+k\beta$ 
    for $\beta\in \Der(\H,U)$, it follows that $\delta_1\in\Der(\H,U)$.
    \smallskip

    We have shown that $\mathtt{P}(\mathcal{C},\delta)$ is an invariant of the \textbf{while} loop 
    that is true entering the loop. Hence $\mathtt{P}(\mathcal{C},\delta)$ is true at the end of 
    the loop when $\mathcal{C}=\mathcal{B}$. It follows that $\delta$ returned by 
    Algorithm~\ref{algo:single-der} behaves as stated.
    \smallskip

     The polynomial complexity is clear, since the number of iterations 
     is bounded by the dimension of $\Der(\H,U)$, which 
     is at most $n$. Observe that the scalar $k\in\K$ on Line 6 can always be 
     found because we must avoid at most $n^2$ values and we assumed that 
     $\K$ is infinite (Remark~\ref{rem:infinite-field}).
\end{proof}

The main purpose of a quotient is to simplify a complicated 
hypergraph. A well chosen quotient should both simplify and 
retain essential features of the original. Given a hypergaph 
$\H$, the quotient $\D(\H)$ in a sense captures features of $\H$ that are
detectable by linear invariants. It is natural to ask what those 
features are, and this is a topic we take up in 
Section~\ref{sec:properties}.
It is also natural to wonder about the dynamics of the function $\D(-)$. 
For example, does repetition   
\[
\H~\rightsquigarrow~\D(\H)~\rightsquigarrow~\D(\D(\H))~\rightsquigarrow~\ldots
\]
yield further simplification?
As it turns out, it does not: $\D(-)$ is a closure function.

\begin{thm}
\label{thm:idempotent}
If $\H$ is a connected $\ell$-uniform hypergraph, then 
$\D(\D(\H))=\D(\H)$.
\end{thm}

\begin{proof}
For $\H=(V,E)\in \mathtt{HYP}_{\ell}$, recall that $\D(\H)=(\bar{V},\bar{E})$ for
\begin{align*}
\bar{V}=\{\,\bar{x} \mid x\in V\,\}, && 
\bar{E} =\{ \bar{e} \mid e\in E\}\subset V^{\ell},
\end{align*}
where $\bar{x}=\{y\in V\mid \forall \delta\in \Der(\H,U),\,\forall a\in [\ell],\;\delta_a(x)=\delta_a(y) \}$ 
and $\bar{e}(a)=\bar{e(a)}$ for $a\in [\ell]$. 
For $\delta\in \Der(\H,U)$, $a\in [\ell]$, and  $x\in V$, 
$\delta_a$ is constant on the class $\bar{x}$.
Hence, each $\delta\in \Der(\H,U)$ induces a function $\bar{\delta}\in (\K^{\bar{V}})^{\ell}$, 
where $\bar{\delta}_a(\bar{x}):=\delta_a(x)$.
Furthermore, for $\bar{e}\in\bar{E}$ and $\sigma\in \Sym(\ell)$, we have
\begin{align*}
    \bar{\delta}_a(\bar{e}^{\sigma}(a)) &= \bar{\delta}_a(\bar{e(a^{\sigma})}) = \delta_a(e(a^{\sigma})),
\end{align*}
so $U(\bar{\delta} * \bar{e}^{\sigma})=U(\delta*e^{\sigma})=0$ and it follows that 
$\bar{\delta}\in\Der(\D(\H),U)$.
In particular, if $\delta\in \Der(\H,U)$ is the element in 
Proposition~\ref{prop:single-der} and $\bar{\delta} \in \Der(\D(\H),U)$,
then each $\bar{\delta}_a$ is one-to-one. Hence, 
the $U$-fusion classes all have size 1, so
$\D(\D(\H))=\D(\H)$.
\end{proof}

\section{Algorithms \& Experiments}
\label{sec:applications}
The main goal of this research is to provide a systematic 
way to simplify hypergraphs using easily computable invariants. 
In this section we provide an illustration of our technique 
in practice. Our quotient algorithm is elementary,
but we present it in Algorithm~\ref{algo:quotient} for easy reference.
A Python implementation of the algorithm is available 
on GitHub~\cite{Python}.

 \begin{algorithm}    
    \caption{(Frame)}
    \begin{algorithmic}[1]
    \Require $\H=(V,E)\in \mathtt{HYP}_{\ell}$.
    \Ensure The frame $\D(\H)\in \mathtt{HYP}_{\ell}$ of $\H$.
    \State $\delta \gets$ generating signal of $\H$ from Algorithm~\ref{algo:single-der} 
    \State $\bar{V} \gets \{\,\bar{x}\,\mid x\in V\,\}$, where $\bar{x}=\{y\in V\mid \delta(x)=\delta(y)\}$
    \State $\bar{E} \gets \{\, \bar{e}\,\mid e\in E\,\}$, where $\bar{e}(a)=\bar{e(a)}$ for $a\in [\ell]$.
    \State\Return $\D(\H)\,=\,(\bar{V},\bar{E}\,)$
    \end{algorithmic}
    \label{algo:quotient}
\end{algorithm}

We tested the algorithm on a range of randomly generated 3-uniform 
hypergraphs to investigate what features of a hypergraph $\H$ are  
detected by signals and preserved in the frame $\D(\H)$. We also
experimented with $\ell$-uniform hypergraphs for $\ell>3$, but 3-uniform 
hypergraphs already exhibit an interesting range of behaviors 
and are easier to visualize.

\begin{ex}
   \label{ex:interesting-hypergraph}
   The 3-uniform hypergraph $\H$ in Figure~\ref{fig:complicated-hypergraph} 
   was randomly generated subject to some sparsity constraints.
   Its frame $\D(\H)$ is depicted in Figure~\ref{fig:quot}.
   Informally, the frame identifies vertices whose \textit{stars} (the set of hyperedges 
   containing the vertex) are of the same general type.

    \begin{figure}[H]
    \centering
    \begin{tikzpicture}[scale=0.6]

    \fill[red!30] (-1.5,2.6) -- (-3.6,0.5) -- (-3.8,-1) -- cycle;
    \fill[green!30] (-1.5,2.6) -- (2.8,1.8) -- (1.5,2.6) -- cycle;
    \fill[blue!30] (2.8,1.8) -- (0,3) -- (1.5,2.6) -- cycle;
    \fill[yellow!30] (-2.2,-3.2) -- (3.8,-1) -- (0,3) -- cycle;
    \fill[orange!30] (-3.8,-1) -- (-0.8,-3.9) -- (-2.2,-3.2) -- cycle;
    \fill[violet!30] (-0.8,-3.9) -- (3.8,-1) -- (2.2,-3.2) -- cycle;
    \fill[cyan!30] (3.2,-2.4) -- (-3.2,-2.4) -- (0.8,-3.9) -- cycle;
    \fill[pink!30] (0.8,-3.9) -- (3.2,-2.4) -- (-2.8,1.8) -- cycle;
    \fill[brown!30] (-2.8,1.8) -- (2.2,-3.2) -- (3.6,0.5) -- cycle;
    \fill[gray!30] (-3.6,0.5) -- (3.6,0.5) -- (-3.2,-2.4) -- cycle;

    \node[circle, draw, fill=white] (1) at (0,3) {1};
    \node[circle, draw, fill=white] (2) at (1.5,2.6) {2};
    \node[circle, draw, fill=white] (3) at (2.8,1.8) {3};
    \node[circle, draw, fill=white] (4) at (3.6,0.5) {4};
    \node[circle, draw, fill=white] (5) at (3.8,-1) {5};
    \node[circle, draw, fill=white] (6) at (3.2,-2.4) {6};
    \node[circle, draw, fill=white] (7) at (2.2,-3.2) {7};
    \node[circle, draw, fill=white] (8) at (0.8,-3.9) {8};
    \node[circle, draw, fill=white] (9) at (-0.8,-3.9) {9};
    \node[circle, draw, fill=white] (10) at (-2.2,-3.2) {10};
    \node[circle, draw, fill=white] (11) at (-3.2,-2.4) {11};
    \node[circle, draw, fill=white] (12) at (-3.8,-1) {12};
    \node[circle, draw, fill=white] (13) at (-3.6,0.5) {13};
    \node[circle, draw, fill=white] (14) at (-2.8,1.8) {14};
    \node[circle, draw, fill=white] (15) at (-1.5,2.6) {15};

    \foreach \a/\b in {15/13, 15/12, 13/12, 15/3, 15/2, 3/2, 3/1, 1/2, 10/5, 10/1, 5/1, 12/9, 12/10, 9/10,
                         9/5, 9/7, 5/7, 6/11, 6/8, 11/8, 8/6, 8/14, 6/14, 14/7, 14/4, 7/4, 13/4, 13/11, 4/11} {
        \draw[thin] (\a) -- (\b);
    }
    \end{tikzpicture}
    \caption{A hypergraph $\H$ on 15 vertices with 10 hyperedges.}
    \label{fig:complicated-hypergraph}
    \end{figure}

    \begin{figure}[H]
    \centering
    \begin{tikzpicture}[scale=0.6]
    \fill[red!30] (0,0) -- (1.5,2.5) -- (3,0) -- cycle;
    \fill[green!30] (3,0) -- (4.5,2.5) -- (6,0) -- cycle;
    \fill[blue!30] (6,0) -- (7.5,2.5) -- (9,0) -- cycle;
    \fill[yellow!30] (9,0) -- (10.5,2.5) -- (12,0) -- cycle;

    \draw[thin] (0,0) -- (1.5,2.5) -- (3,0) -- cycle;
    \draw[thin] (3,0) -- (4.5,2.5) -- (6,0) -- cycle;
    \draw[thin] (6,0) -- (7.5,2.5) -- (9,0) -- cycle;
    \draw[thin] (9,0) -- (10.5,2.5) -- (12,0) -- cycle;

    \node[rectangle, draw, fill=white] (1) at (0,0) {$\{2\}$};
    \node[rectangle, draw, fill=white] (2) at (1.5,2.5) {$\{3\}$};
    \node[rectangle, draw, fill=white] (3) at (3,0) {$\{1,9,15\}$};
    \node[rectangle, draw, fill=white] (4) at (4.5,2.5) {$\{5,12\}$};
    \node[rectangle, draw, fill=white] (5) at (6,0) {$\{7,10,13\}$};
    \node[rectangle, draw, fill=white] (6) at (7.5,2.5) {$\{4\}$};
    \node[rectangle, draw, fill=white] (7) at (9,0) {$\{11,14\}$};
    \node[rectangle, draw, fill=white] (8) at (10.5,2.5) {$\{6\}$};
    \node[rectangle, draw, fill=white] (9) at (12,0) {$\{8\}$};
    \end{tikzpicture}
    \caption{The frame $\D(\H)$ of the hypergraph $\H$ in 
    Figure~\ref{fig:complicated-hypergraph}.}
    \label{fig:quot}
    \end{figure}
   
\end{ex}

For moderately sparse hypergraphs the frame  
provides a meaningful simplification.  
When a hypergraph is too dense, many vertices collapse into a single equivalence class in the quotient, 
often resulting in near-complete reduction. When it is too sparse, on the other hand,
we tend to get very small fusion classes (often of size 1), yielding little to no reduction.

\begin{figure}[H]
    \centering
    \begin{tikzpicture}[scale=0.80]
        \begin{axis}[
            width=12cm,
            height=8cm,
            xlabel={Average number of hyperedges per vertex},
            ylabel={Reduction proportion (avg over 50 runs)},
            title={Reduction Proportion vs.\ Average Hyperedges per Vertex (Varying Hypergraph Sizes)},
            xmin=2.25, xmax=3.05,
            ymin=0, ymax=1,
            xtick={2.3,2.4,2.5,2.6,2.7,2.8,2.9,3.0},
            ytick={0,0.2,0.4,0.6,0.8,1.0},
            grid=major,
            axis line style={thick},
            tick label style={font=\small},
            label style={font=\small},
            title style={font=\small},
            legend style={at={(0.5,-0.2)}, anchor=north, draw=none, font=\small, legend columns=-1}
        ]

        \addplot[
            color=red!70!black,
            mark=*,
            mark size=2pt,
            line width=1pt
        ] coordinates {
            (2.3,0.78)
            (2.4,0.73)
            (2.5,0.58)
            (2.6,0.63)
            (2.7,0.38)
            (2.8,0.19)
            (2.9,0.09)
            (3.0,0.03)
        };
        \addlegendentry{50 vertices}

        \addplot[
            color=blue!70!black,
            mark=square*,
            mark size=2pt,
            line width=1pt
        ] coordinates {
            (2.3,0.91)
            (2.4,0.91)
            (2.5,0.88)
            (2.6,0.76)
            (2.7,0.60)
            (2.8,0.31)
            (2.9,0.02)
            (3.0,0.02)
        };
        \addlegendentry{100 vertices}

        \addplot[
            color=green!60!black,
            mark=triangle*,
            mark size=2.5pt,
            line width=1pt
        ] coordinates {
            (2.3,0.95)
            (2.4,0.94)
            (2.5,0.94)
            (2.6,0.87)
            (2.7,0.63)
            (2.8,0.21)
            (2.9,0.02)
            (3.0,0.02)
        };
        \addlegendentry{150 vertices}

        \addplot[
            color=purple!70!black,
            mark=diamond*,
            mark size=2.5pt,
            line width=1pt
        ] coordinates {
            (2.3,0.96)
            (2.4,0.96)
            (2.5,0.96)
            (2.6,0.96)
            (2.7,0.77)
            (2.8,0.10)
            (2.9,0.04)
            (3.0,0.02)
        };
        \addlegendentry{200 vertices}

        \end{axis}
    \end{tikzpicture}
    \caption{Reduction proportion as a function of average hyperedges per vertex for hypergraphs of various sizes.}
    \label{fig:lineplot}
\end{figure}

To quantify this behavior, let us define the \textit{reduction proportion} 
of hypergraph $\H$ to be the ratio of the number of hyperedges 
in the frame $\D(\H)$ to the number of hyperedges in $\H$.
The hypergraphs we encounter in the wild are presumably not random, but
it is still instructive to see how the reduction proportion behaves
on randomly generated hypergraphs of varying sizes and sparsities.
We generated random 3-uniform hypergraphs with vertex counts of 
50, 100, 150, and 200, varying the number of hyperedges to
achieve average hyperedge-per-vertex counts from 2.3 to 3.0.
For each combination of vertex count and average hyperedges per vertex,
we generated 50 random hypergraphs and computed the average reduction proportion.
Figure~\ref{fig:lineplot} shows how this reduction proportion varies with the average number of hyperedges per 
vertex for hypergraphs of different sizes.
The figure suggests that meaningful reduction occurs only within a narrow 
sparsity window.

\section{Properties of derivation quotients}
\label{sec:properties}
We turn now to the question of what the function $\D(\H)$ does 
to a hypergraph $\H$. Let us start by revisiting our favorite 
5-vertex hypergraph in
Example~\ref{ex:5-vertex-revisited}. 
Physically, one can think of the hypergaph $\D(\H)$ on 3 vertices
as being obtained from $\H$ 
by a 2-step folding process: first fold vertex $y$ onto vertex $w$ 
along the join $ux$, and then fold $x$ onto $v$ along join $uw$.
The process is depicted in Figure~\ref{fig:5-vertex-folding}.

\begin{figure}[H]
    \centering
    \begin{subfigure}[b]{0.3\textwidth} 
    \centering
    \begin{tikzpicture}[scale = 0.6]
    \fill[red!20] (90:2) -- (162:2) -- (234:2) -- cycle;   
    \fill[green!20] (90:2) -- (234:2) -- (306:2) -- cycle;  
    \fill[blue!20] (90:2) -- (306:2) -- (18:2) -- cycle;    

    \node[circle, draw, fill=white] (1) at (90:2) {$u$};
    \node[circle, draw, fill=white] (2) at (162:2) {$v$};
    \node[circle, draw, fill=white] (3) at (234:2) {$w$};
    \node[circle, draw, fill=white] (4) at (306:2) {$x$};
    \node[circle, draw, fill=white] (5) at (18:2) {$y$};

   \draw[thin] (1) -- (2);  
   \draw[thin] (2) -- (3);  
   \draw[thin] (3) -- (4);  
   \draw[thin] (4) -- (5);  
   \draw[thin] (5) -- (1);  
   \draw[thin] (1) -- (3);  
   \draw[very thick, dotted] (1) -- (4);  
    \end{tikzpicture}
    \end{subfigure}
    \hfill
    \begin{subfigure}[b]{0.3\textwidth}
        \centering
        \begin{tikzpicture}[scale = 0.6]
    \fill[red!20] (90:2) -- (162:2) -- (234:2) -- cycle;   
    \fill[cyan!20] (90:2) -- (234:2) -- (306:2) -- cycle;  

    \node[circle, draw, fill=white] (1) at (90:2) {$u$};
    \node[circle, draw, fill=white] (2) at (162:2) {$v$};
    \node[circle, draw, fill=white] (3) at (234:2) {$w$};
    \node[circle, draw, fill=white] (4) at (306:2) {$x$};

   \draw[thin] (1) -- (2);  
   \draw[thin] (2) -- (3);  
   \draw[thin] (3) -- (4);  
   \draw[very thick, dotted] (1) -- (3);  
   \draw[thin] (1) -- (4);  
    \end{tikzpicture}
    \end{subfigure}
    \hfill
    \begin{subfigure}[b]{0.3\textwidth} 
        \centering
        \begin{tikzpicture}[scale=0.6]
        \fill[orange!20] (90:2) -- (210:2) -- (330:2) -- cycle;

        \node[circle, draw, fill=white, minimum size=0.5cm] (A) at (90:2) {$\bar{u}$};
        \node[circle, draw, fill=white, minimum size=0.5cm] (B) at (210:2) {$\bar{v}$};
        \node[circle, draw, fill=white, minimum size=0.5cm] (C) at (330:2) {$\bar{w}$};
    
        \draw[thin] (A) -- (B);
        \draw[thin] (B) -- (C);
        \draw[thin] (C) -- (A);
    \end{tikzpicture}
    \end{subfigure}
    
    \caption{On the left is our familiar hypergraph $\mathcal{H}$, 
    in the middle is $\H$ folding along the ``join" $u$--$x$, 
    and on the right is $\D(\mathcal{H})$ obtained 
    from the intermediate hypergraph by folding along 
    the ``join'' $u$--$w$.}
    \label{fig:5-vertex-folding}
\end{figure}
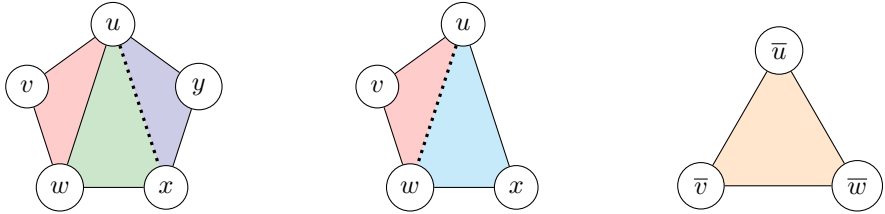

It turns out that this folding mechanism in the frame of a hypergraph  
is the consequence of a general phenomenon:

\begin{prop}
\label{prop:folding}
    Let $\H = (V,E)\in\mathtt{HYP}_{\ell}$. For distinct $x,y\in V$,
    suppose there are edges $e,f\in E$ and $\sigma\in\Sym(\ell)$ satisfying 
    \begin{align*}
    e^{\sigma}(1)=x, \quad f(1)=y, \quad e^{\sigma}(a)=f(a)~\mbox{for}~a\in \{2,\ldots,\ell\}.
    \end{align*}
    Then $x\equiv_Uy$, so 
    $\bar{x}=\bar{y}$ in the frame of $\H$.
\end{prop}

\begin{proof}
    This is immediate from the definition of signal. 
    For $\delta\in\Der(\H,U)$,
    since $U(\delta*e^{\sigma})=0=U(\delta*f)$, 
    we have $\sum_{a\in [\ell]}\delta_a(e^{\sigma}(a))=\sum_{a\in [\ell]}\delta_a(f(a))$. 
    It follows from our assumptions on $e,f$ and $\sigma$ that 
    $\delta_1(x)=\delta_1(y)$ and hence that $\delta_a(x)=\delta_a(y)$ for 
    all $a\in [\ell]$, so $x\equiv_Uy$. 
\end{proof}

It is reasonable to ask whether $\D(\H)$ can always be obtained 
from $\H$ using a folding process. On the one hand this would provide 
a neat characterization for what $\D(-)$ is doing, but on the other 
hand it would also render the mechanics of the frame function somewhat 
dull. Fortunately,
this is not the whole story, even for 3-uniform hypergaphs, as 
Figure~\ref{fig:unstable-config} demonstrates. 

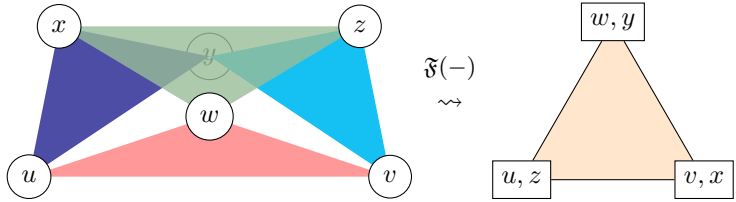
\begin{figure}[H]
    \centering

    \begin{subfigure}{0.4\textwidth} 
        \centering
        \begin{tikzpicture}[scale=0.8]

\node[circle, draw, fill=white] (5) at (0,2) {$y$};

\fill[blue!70] (-3,0) -- (0,2) -- (-2.5,2.5) -- cycle;

\fill[cyan!70] (3,0) -- (0,2) -- (2.5,2.5) -- cycle;

\fill[red!40] (-3,0) -- (3,0) -- (0,1) -- cycle;

\fill[green!70!black!40,opacity=0.8] (0,1) -- (-2.5,2.5) -- (2.5,2.5) -- cycle;

\node[circle, draw, fill=white] (1) at (-3,0) {$u$};
\node[circle, draw, fill=white] (2) at (3,0) {$v$};
\node[circle, draw, fill=white] (3) at (0,1) {$w$};
\node[circle, draw, fill=white] (4) at (-2.5,2.5) {$x$};
\node[circle, draw, fill=white] (6) at (2.5,2.5) {$z$};
\end{tikzpicture}
    \end{subfigure}
    \raisebox{0.5in}[0pt][0pt]{
    \begin{tikzpicture}[remember picture, overlay, baseline=(current bounding box.center)]
        \node (arrow_center) at (0.6, 0) {$\rightsquigarrow$};
        \node[above=0pt of arrow_center, font=\small] {$\mathfrak{F}(-)$};
    \end{tikzpicture}
    }
    \begin{subfigure}{0.4\textwidth} 
        \centering
        
        \begin{tikzpicture}[scale=0.7]
        \fill[orange!20] (90:2) -- (210:2) -- (330:2) -- cycle;

        \node[rectangle, draw, fill=white, minimum size=0.5cm] (A) at (90:2) {$w,y$};
        \node[rectangle, draw, fill=white, minimum size=0.5cm] (B) at (210:2) {$u,z$};  
        \node[rectangle, draw, fill=white, minimum size=0.5cm] (C) at (330:2) {$v,x$};  
    
        \draw[thin] (A) -- (B);
        \draw[thin] (B) -- (C);
        \draw[thin] (C) -- (A);
    \end{tikzpicture}
    \end{subfigure}

    \caption{On the left is a 3-uniform hypergraph $\H$ on 6 vertices with 4 edges. On 
the right is its frame $\D(\H)$. This shows that the operator $\D(-)$ 
does not simply ``fold up" a given 3-uniform hypergraph along 
the common ``joins'' of its edges.}
    \label{fig:unstable-config}
\end{figure}

The function $\D$ can be interpreted as a projection onto 
a subset $\mathtt{STAB}_{\ell}$ 
of hypergraphs that are \textit{stable} under $\D$, namely
\begin{align}
\label{eq:stable}
\mathtt{STAB}_{\ell}(\D) &:= \{\,\H\in\mathtt{HYP}_{\ell} \mid \D(\H)=\H\,\}.
\end{align}
By Theorem~\ref{thm:idempotent},
the set $\mathtt{STAB}_{\ell}(\D)$ is characterized as the 
image of $\D$, but can we recognize when a given $\H$ belongs 
to $\mathtt{STAB}_{\ell}(\D)$ without computing $\D(\H)$? 
This seems a difficult task in general, but the 
following result shows how one can easily identify some stable 
hypergraphs. 

\begin{prop}
\label{prop:new-from-old}
    Let $\H=(V,H)\in\mathtt{STAB}_{\ell}(\D)$ and let $z\in V$. 
    Let $\hat{\H}=(\hat{V},\hat{H})$ be any $\ell$-uniform hypergraph, where 
    $\hat{V}=V\,\dot{\sqcup}\,W$ with $|W|>1$ 
    and $\hat{E}=E\,\dot{\sqcup}\,\{\hat{e}\}$ with $\hat{e}(1)=z$ and $\hat{e}(a)\in W$ 
    for $a\in \{2,\ldots,\ell\}$. Then $\hat{\H}\in\mathtt{STAB}_{\ell}(\D)$.
\end{prop}

\begin{proof}
Since $\H=(V,H)\in\mathtt{STAB}_{\ell}(\D)$,
by Proposition~\ref{prop:single-der}
there exists $\delta\in\Der(\H,U)$ with each $\delta_a:V\to \K$
a one-to-one function. By the converse of that same result 
it suffices to exhibit $\hat{\delta}\in\Der(\hat{\H},U)$ 
with $\hat{\delta}_1$ one-to-one.

For each $a\in [\ell]$, define $\hat{\delta}_a:\hat{V}\to \K$ as follows. 
For $x\in V$, put $\hat{\delta}_a(x)=\delta_a(x)$. 
This ensures that $\hat{\delta}$ satisfies 
condition~\eqref{eq:derivation} for all $e\in E$. 

We must now extend $\hat{\delta}_a$ to $W$ so that condition~\eqref{eq:derivation}
holds for the new edge $\hat{e}$, namely
$\ch(\hat{\delta}*\hat{e}^{\sigma})=0$ 
for each $\sigma\in\mathrm{Sym}(\ell)$. 
By Lemma~\ref{lem:derivation-constant}, for each $a\in [\ell]$ there exists 
$k_a\in\K$ such that $\delta_a(x)-\delta_1(x)=k_a$ for all $x\in V$, so
\begin{align*}
U(\hat{\delta}*\hat{e}^{\sigma}) &= \sum_{a=1}^{\ell}\hat{\delta}_a(\hat{e}^{\sigma}(a))
=\sum_{a=1}^{\ell}k_a~+~\sum_{a=1}^{\ell}\hat{\delta}_1(\hat{e}^{\sigma}(a)).
\end{align*}
Since the multiset of values $\hat{e}^{\sigma}(a)$ as $a$ ranges over $[\ell]$ 
is independent of $\sigma$, we may assume that $\sigma$ is the identity.
Putting $k=\delta_a(1)+\sum_{a\in [\ell]}k_a$, we therefore require 
\begin{align} 
    \label{eq:der-hat}
    k+\sum_{a=2}^{\ell}\hat{\delta}_1(\hat{e}(a)) &= 0.
\end{align}
Fixing $w\in W$, we can freely choose the values of $\hat{\delta}_1(y)$ 
for all $y\in W-\{w\}$ and then assign $\hat{\delta}_1(w)$ so that 
equation~\eqref{eq:der-hat} holds. Furthermore, since $W-\{w\}$ is 
nonempty and $\K$ is infinite, the values of $\hat{\delta}_1$ on $W$ 
can be chosen so they are not in $\{\delta_1(x)\mid x\in V\}$ 
and are distinct from each other. 
\smallskip

We can therefore construct $\hat{\delta} \in\Der(\hat{\H},U)$ 
with $\hat{\delta}_1$ a one-to-one function. It follows that 
$\hat{\H}\in\mathtt{STAB}_{\ell}(\D)$, as required.
\end{proof}

\begin{rem}
    \label{rem:stable-two-ways}
Let us revisit the 3-uniform hypergraph $\H$ and its frame $\D(\H)$ from 
Example~\ref{ex:interesting-hypergraph}. Since $\D(\D(\H))=\D(\H)$ 
by Theorem~\ref{thm:idempotent}, we already know that 
$\D(\H)\in \mathtt{STAB}_{\ell}(\D)$. Proposition~\ref{prop:new-from-old}
provides another way to see this.
Indeed, the ``triangle'' hypergraph on 
3 vertices with a single edge containing all vertices is stable (cf.~Example~\ref{ex:U-signal}).
Using Proposition~~\ref{prop:new-from-old} we can successively glue 
new triangles onto this hypergraph, adding two new vertices and one new 
hyperedge each time, to create stable ``mountain range'' hypergraphs of the 
sort depicted in Figure~\ref{fig:quot}.
\end{rem}

\begin{rem}
   \label{rem:extensions}
   One can extend Proposition~\ref{prop:new-from-old} in various ways 
   when $\ell>3$. For example, one can show that 
   4-uniform graphs $\H$ such as 
   \medskip

   
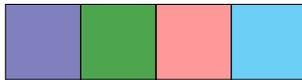
\begin{center}
   \begin{tikzpicture}
\definecolor{sq1}{RGB}{255, 100, 100}   
\definecolor{sq2}{RGB}{100, 200, 255}   
\definecolor{sq3}{RGB}{150, 255, 150}   
\definecolor{sq4}{RGB}{255, 220, 120}   

\fill[blue!50] (0,0) rectangle (1,1);
\fill[green!70] (1,0) rectangle (2,1);
\fill[red!40] (2,0) rectangle (3,1);
\fill[cyan!50] (3,0) rectangle (4,1);

\draw[black] (0,0) rectangle (1,1);
\draw[black] (1,0) rectangle (2,1);
\draw[black] (2,0) rectangle (3,1);
\draw[black] (3,0) rectangle (4,1);

\end{tikzpicture}
\end{center}
\medskip
   
\noindent are stable by proving that new 
   squares can be glued along an edge of an existing 
   square in a stable hypergraph and remain stable.
\end{rem}

\section{Closing remarks}
\label{sec:closing}
In Remark~\ref{rem:comb-v-gen} we mentioned that signals 
are restricted versions of the derivations we could possible 
use if we treated the given hypergraph as a tensor. Recall 
that there are two reasons for working with these restricted 
derivations. The first 
is that it is easier to interpret the outcomes from a combinatorial 
perspective. The second, as noted in Remark~\ref{rem:computing-der}, 
is that it is much faster to compute 
signals than derivations.
Indeed, the specific goal of this paper was to introduce linear 
invariants of a hypergraph that can be quickly computed, and which 
capture key features of the hypergraph. While there remain questions 
to be explored, the main results of the paper achieve the 
objectives that motivated the work.

However, the signals used in this paper 
do not detect features of the hypergraph that are detectable 
by derivations. The space of derivations of a tensor 
forms a Lie algebra~\cite[Theorem 8.2]{Strata} and preliminary
experiments with tensors associated to hypergraphs strongly suggest
that in many cases the structure of this Lie algebra can be used 
to recover features of the hypergraph that our methods miss 
entirely~\cite{M-W}. We mention two such families of 3-uniform hypergraphs. 

First, the hypergaph $\H$ in Example~\ref{ex:5-vertex} belongs 
to a family $\mathcal{F}_n$ of ``fans'' parameterized by the 
number $n$ of segments in the fan. Thus, $\H=\mathcal{F}_3$ because $\H$ has 3 segments.
From Proposition~\ref{prop:folding} we see that $\D(\mathcal{F}_n)=\mathcal{F}_1$ 
for all $n$ (we just fold up all of the segments). Thus, our 
frame function $\D(-)$ does not distinguish between individual members 
in the family $\mathcal{F}_n$, except insofar as the equivalence 
classes of vertices grow in size. On the 
other hand, the Lie algebra of derivations of $\mathcal{F}_n$ grows
in a predictable way as $n$ increases. More precisely, the 
derivations of $\mathcal{F}_n$ are representations of a family 
of simple Lie algebras. 

Secondly, the frame 
$\D(\H)$ in Example~\ref{fig:quot} belongs to a family $\mathcal{M}_n$ of  
``mountain ranges''. It follows from Proposition~\ref{prop:new-from-old}
that $\mathcal{M}_n\in\mathtt{STAB}_3(\D)$. In particular, $\D(\mathcal{M}_n)$ 
does not distinguish the peaks of the mountains from their bases.
On the other hand, these features can (with some effort) be recovered 
from the Lie algebra structure of its derivations.

Thus, the application of the \textit{full} invariants introduced 
in~\cite{Strata} to the \textit{tensors} representing hypergraphs 
provides a fertile area for future research.

\subsection*{Acknowledgments}

P.A. Brooksbank was supported by 
National Science Foundation Grant DMS 2319372.
The authors are grateful to Todd Schmid and Lucas 
Waddell at Bucknell University for fruitful 
discussions in the development of this work, 
and to James Wilson at Colorado State University 
for some helpful suggestions 
on an early draft of the article.

\bibliographystyle{abbrv}

\begin{bibdiv}
\begin{biblist}

\bib{BS-Symmetry}{article}{
   author={Banerjee, Anirban},
   author={Parui, Samiron},
   title={Symmetries of hypergraphs and some invariant subspaces of matrices
   associated with hypergraphs},
   journal={Linear Algebra Appl.},
   volume={726},
   date={2025},
   pages={328--358},
   issn={0024-3795},
   review={\MR{4941857}},
   doi={10.1016/j.laa.2025.07.030},
}

\bib{BS-Building-Blocks}{article}{
   author={Banerjee, Anirban},
   author={Parui, Samiron},
   title={On some building blocks of hypergraphs},
   journal={Linear Multilinear Algebra},
   volume={73},
   date={2025},
   number={7},
   pages={1369--1402},
   issn={0308-1087},
   review={\MR{4898769}},
   doi={10.1080/03081087.2024.2417213},
}

\bib{Berge-Book}{book}{
   author={Berge, Claude},
   title={Hypergraphs},
   series={North-Holland Mathematical Library},
   volume={45},
   note={Combinatorics of finite sets;
   Translated from the French},
   publisher={North-Holland Publishing Co., Amsterdam},
   date={1989},
   pages={x+255},
   isbn={0-444-87489-5},
   review={\MR{1013569}},
}

\bib{OpenDleto}{webpage}{
    author={Brooksbank, Peter A.},
    author={Kassabov, Martin D.},
    author={Wilson, James B.},
    title={OpenDleto},
    subtitle={Algorithms for structure recovery in tensors},
    url={https://github.com/thetensor-space/OpenDleto},
    date={}
}

\bib{Strata}{article}{
    author={Brooksbank, Peter A.},
    author={Kassabov, Martin D.},
    author={Wilson, James B.},
    title={Detecting null patterns in tensor data},
    url={https://arxiv.org/abs/2408.17425},
    date={}
}

\bib{Python}{webpage}{
    author={Brooksbank, Peter A.}
    author={Chaplin, Clara R.},
    title={HyperQuotients},
    subtitle={Quotients of hypergraphs from linear invariants},
    url={https://github.com/crc035/Hypergraph-Linear-Invariants},
    date={}
}

\bib{Chung-Book}{book}{
    author={Chung, F.},
    title={Spectral Graph Theory},
    series    = {CBMS Regional Conference Series in Mathematics},
    volume    = {92},
    year      = {1997},
    publisher = {American Mathematical Society},
    address   = {Providence, RI}
}

\bib{CD-Hypergraph-Spectra}{article}{
   author={Cooper, Joshua},
   author={Dutle, Aaron},
   title={Spectra of uniform hypergraphs},
   journal={Linear Algebra Appl.},
   volume={436},
   date={2012},
   number={9},
   pages={3268--3292},
   issn={0024-3795},
   review={\MR{2900714}},
   doi={10.1016/j.laa.2011.11.018},
}

\bib{FMW}{article}{
    title={A spectral theory for transverse tensor operators},
    author={First, Uriya},
    author={Maglione, Joshua},
    author={Wilson, James B.},
    note={arXiv:1911.02518}
    year={2020},
}

\bib{KB-tensor-decomp}{article}{
   author={Kolda, Tamara G.},
   author={Bader, Brett W.},
   title={Tensor decompositions and applications},
   journal={SIAM Rev.},
   volume={51},
   date={2009},
   number={3},
   pages={455--500},
   issn={0036-1445},
   review={\MR{2535056}},
   doi={10.1137/07070111X},
}

\bib{M-W}{article}{
   author={Mini\~{n}o, Amaury},
   author={Wilson, James B.},
   title={Private communication} 
}

\bib{PZ-hypergraph-tensor}{article}{
   author={Pearson, Kelly J.},
   author={Zhang, Tan},
   title={On spectral hypergraph theory of the adjacency tensor},
   journal={Graphs Combin.},
   volume={30},
   date={2014},
   number={5},
   pages={1233--1248},
   issn={0911-0119},
   review={\MR{3248502}},
   doi={10.1007/s00373-013-1340-x},
}

     
\end{biblist}
\end{bibdiv}

\end{document}